\documentclass[11pt]{article}
\usepackage{amsthm}
\usepackage{amsmath,amssymb,  amscd,eucal}
 \usepackage{latexsym}
\usepackage{graphicx}
\usepackage[usenames]{color}
\usepackage{times,tikz}
\usepackage{paralist}
\usetikzlibrary{decorations.markings}

\usetikzlibrary{shapes.geometric,shapes.arrows,decorations.pathmorphing}
\usetikzlibrary{matrix,chains,scopes,positioning,arrows,fit}
\tikzset{
  treenode/.style = {align=center, inner sep=0pt, text centered,
    font=\sffamily},
  arn_n/.style = {treenode, circle, white, font=\sffamily\bfseries, draw=black,
    fill=black, text width=1.5em},
  arn_r/.style = {treenode, circle, red, draw=red,
    text width=1.5em, very thick},
  arn_x/.style = {treenode, rectangle, draw=black,
    minimum width=0.5em, minimum height=0.5em}
}

\newtheorem{thm}[equation]{Theorem}
\newtheorem{cor}[equation]{Corollary}

\newtheorem{prop}[equation]{Proposition}
\newtheorem{para}[equation]{}

\newtheorem{remark}[equation]{Remark}
\newtheorem{example}[equation]{Example}

\numberwithin{equation}{section}

\newcommand{\angstrom}{\mbox{\normalfont\AA}}

\renewcommand{\det}{\mathsf{det}}
  \newcommand{\diag}{\mathsf {diag}}

\newcommand{\ZZ}{\mathbb{Z}}
\newcommand{\CC}{\mathbb{C}}
\newcommand{\cE}{\mathcal{E}}
\newcommand{\arm}{\mathsf{a}}
\newcommand{\br}{\mathsf{b}}
\newcommand{\cm}{\mathsf{c}}
\newcommand{\er}{\mathrm{e}}
\newcommand{\gr}{\mathrm{g}}

\newcommand{\Ar}{\mathrm{A}}

\newcommand{\Sr}{\mathrm{S}}
\newcommand{\Tr}{\mathrm{T}}

\newcommand{\al}{\alpha}
\newcommand{\lam}{\lambda}
\newcommand{\ve}{\varepsilon}
\newcommand{\es}{\mathsf{e}}

\newcommand{\ms}{\mathsf{m}}
\newcommand{\ps}{\mathsf{p}}

\newcommand{\EE}{\mathsf{E}}
\newcommand{\GG}{\mathsf{G}}

\newcommand{\TT}{\mathsf{T}}

\newcommand{\VV}{\mathsf{V}}
\newcommand{\XX}{\mathsf{X}}
 \newcommand{\SU}{\mathsf{SU}}

\newcommand{\Zs}{\mathsf{Z}}
\newcommand{\ZsZ}{\Zs_k(\ZZ_2^n)}
\newcommand{\ZsG}{\Zs_k\big(\GG(2,1,n)\big)}
\newcommand{\ZsS}{\Zs_k(\Sr_n)}

\newcommand\End{\mathsf {End}}

\renewcommand{\cosh}{\mathsf{cosh}}
\renewcommand{\sinh}{\mathsf{sinh}}

\newcommand{\binombr}[2]{\genfrac{\{}{\}}{0pt}{}{#1}{#2} }

\def \ot {\otimes}
\def\modd{\, \mathsf{mod} \,}

\def\dimm{\, \mathsf{dim}\,}

\begin{document}
\title{A Schur-Weyl Duality Approach to \\ Walking on Cubes}
\author{Georgia Benkart and Dongho Moon\thanks{This research was supported by the Basic Science Research Program of the National Research Foundation of Korea (NRF) funded by the Ministry of Education, Science and Technology (2010-0022003).
The hospitality of the Mathematics Department at the University of Wisconsin-Madison while this research
was done  is gratefully
acknowledged.}} \date{}
\maketitle
\vspace{-.8cm}
\begin{center} {\it In memory of Professor Hyo Chul Myung}
\end{center}

\begin{abstract} Walks on the representation graph $\mathcal R_{\VV}(\GG)$ determined
by a group $\GG$ and a $\GG$-module $\VV$ are related to the centralizer algebras
of the action of $\GG$ on the tensor powers $\VV^{\ot k}$  via Schur-Weyl duality.
This paper explores that connection when the group is $\ZZ_2^n$ and the module $\VV$ is
chosen so the representation graph is the $n$-cube.  We describe a basis for the centralizer
algebras in terms of labeled partition diagrams.  We obtain an expression for the number
of walks  by counting certain partitions  and determine the exponential generating functions for the number of walks.  \end{abstract}
\textbf{MSC Numbers (2010)}:  05E10, 20C05   \hfill \newline

\begin{section} {Introduction} \end{section}

Walks on graphs have widespread applications in modeling networks,
particle interactions,  biological and random processes,  and many other phenomena. Typically, the walker (an impulse,
physical or biological quantity, or person) transitions from one node to another along an edge, which may have an assigned probability.   Some natural questions  that arise in this context are:
 How many  different walks of $k$ steps are there from
node $\arm$  to node $\br$ on the graph?  What is the probability that a particle moves from $\arm$ to $\br$  in
$k$ steps?

 Walks on graphs are also related to  chip-firing games,   or to what is often referred to
in physics  as the (abelian) sandpile model.
In a chip-firing game, each node starts with  a pile of chips.
A step consists of selecting a node with at least as many chips as its degree
and moving one chip  from that node to each of its adjacent neighbors.  The game continues
indefinitely or  terminates when no more firings are possible.  In the latter case,
 the number of steps is related to the least positive eigenvalue of
the Laplace operator of the graph \cite{BLS},  and it is bounded by an expression in
the Dirichlet eigenvalues \cite{CE}.

The graphs considered here  arise from the representation theory of groups  in the following way:   Let $\GG$ be a finite group and $\VV$  be a finite-dimensional $\GG$-module over the complex field $\CC$.
The representation graph $\mathcal{R}_\VV(\GG)$  of $\GG$ associated to $\VV$  has
nodes corresponding to the irreducible $\GG$-modules $\{\GG^\lambda \mid \lam \in \Lambda(\GG)\}$ over $\CC$.
For $\mu  \in \Lambda(\GG)$, there are $a_{\mu,\lam}$  edges
from $\mu$ to $\lam$ in  $\mathcal{R}_\VV(\GG)$  if
$$\GG^\mu \ot \VV =  \bigoplus_{\lam \in \Lambda(\GG)} a_{\mu,\lam}  \GG^\lam.$$
  Thus, the number of edges $a_{\mu,\lam}$ from $\mu$ to  $\lam$ in  $\mathcal{R}_\VV(\GG)$ is the multiplicity of $\GG^\lambda$  as a summand of $\GG^\mu \ot \VV$.

Let $\GG^{\mathbf 0}$ be the trivial one-dimensional $\GG$-module on which every element of $\GG$
acts as the identity transformation.  Since each step on the graph is achieved by tensoring with $\VV$,
\begin{align*} \ms_k ^\lam: &= \text{number  of walks of $k$ steps from $\mathbf 0$  to} \ \lam \\
& = \text{multiplicity of  $\GG^\lam$   in   $\GG^{\mathbf 0} \ot  \VV^{\ot k}  \cong \VV^{\ot k}$.}
\end{align*}
When the action of $\GG$ on $\VV$ is faithful, then every irreducible $\GG$-module $\GG^{\lam}$ occurs
in some $\VV^{\ot k}$.

The {\it centralizer algebra},
\begin{equation}\label{eq:cent} \Zs_k(\GG) = \{z \in \End(\VV^{\ot k}) \mid  z(g.w) = g.z(w) \  \ \forall \ g \in \GG,  w \in \VV^{\ot k}\}, \end{equation}
plays an essential role in studying $\VV^{\ot k}$, as it contains the projection maps onto  the irreducible summands of $\VV^{\ot k}$.

Let $\Lambda_k(\GG)$
denote the subset of $\Lambda(\GG)$ corresponding to the  irreducible $\GG$-modules which occur in $\VV^{\ot k}$ with
multiplicity at least one.    \emph{Schur-Weyl duality}
establishes important connections between the representation theories of $\GG$ and $\Zs_k(\GG)$:
 \begin{itemize}
  \item  the irreducible $\Zs_k(\GG)$-modules are in bijection with the elements of $\Lambda_k(\GG)$;
 \item   the $\GG$-module decomposition of \ $\VV^{\ot k}$ into irreducible summands  is given by
$\VV^{\ot k}  \,\cong \,  \bigoplus_{\lambda \in \Lambda_k(\GG)} \ms_k^\lambda \GG^\lambda,$  where \ $\ms_k^\lambda$ is the number of walks of  $k$   steps from  $\mathbf 0$  to  \ $\lam$
on the representation graph $\mathcal{R}_\VV(\GG)$;
  \item  the  $\Zs_k(\GG)$-module decomposition of \ $\VV^{\ot k}$ into irreducible $\Zs_k(\GG)$-modules
  $\Zs_k^\lam$, $\lam \in \Lambda_k(\GG)$,  is given by  \
  $\VV^{\ot k} \, \cong \,  \bigoplus_{\lambda \in \Lambda_k(\GG)}
  \mathsf {d}^\lambda \,  \Zs_k^\lambda$,  where
$$\mathsf{d}^\lambda = \dimm \GG^\lambda \quad  \text{and} \quad \ms_k^\lambda = \dimm \Zs_k^\lambda.$$
  \item $ \dimm\Zs_k(\GG) = \sum_{\lambda \in \Lambda_k(\GG)}  (\dim \, \Zs_k^\lambda)^2 =  \sum_{\lambda \in \Lambda_k(\GG)}  (\ms_k^\lambda)^2  \ = \ \ms_{2k}^{\mathbf{0}}$, \\
(the number  of walks of \ $2k$  \ steps from  $\mathbf{0}$  to \ $\mathbf{0}$ \  on \  ${\mathcal R}_\VV(\GG)$).
\end{itemize}

The following result, which was shown in \cite{B},
gives an efficient way of  computing the Poincar\'e series
$$\ms^\mu(t) = \sum_{k \geq 0}  \ms_k^\mu t^k$$
for the multiplicities $\ms_k^\mu$, $\mu \in \Lambda(\GG)$
(that is,  for the number of walks of $k$ steps from $\mathbf 0$ to $\mu$ on the
representation graph $\mathcal R_{\VV}(\GG)$, or equivalently,  for the dimension
of the centralizer modules $\Zs_k^\mu$, $k \geq 0$).  We assume $\VV^{\ot 0} = \GG^{\mathbf 0}$ and  the columns of the adjacency 
matrix have been indexed so that the one corresponding to $\mathbf 0$ is the first. 
\medskip

\begin{thm}\label{T:Poin}{\rm ([B, Thm.~2.1])} \ Let $\GG$ be a finite group with irreducible modules $\GG^\lam$, $\lambda \in \Lambda(\GG)$,  over $\CC$, and let $\VV$ be a finite-dimensional $\GG$-module such  that the action of $\GG$ on $\VV$ is faithful.  Assume $\ms^\mu(t) = \sum_{k \geq 0} \ms_k^\mu t^k$ is the Poincar\'e series for the multiplicities $\ms_k^\mu$ ($k \geq 0$)  of
$\GG^\mu$ in  $\mathsf{T}(\VV) = \bigoplus_{k \geq 0} \VV^{\ot k}$.
Let $\mathrm{A}$ be the adjacency matrix of the
representation graph $\mathcal R_{\VV}(\GG)$,  and let  $\mathrm{M}^\mu$ be the
matrix $\mathrm{I}-t\mathrm{A}$ with the column indexed by $\mu$ replaced by
$\underline{\delta} = \left (\begin{smallmatrix}  1\\ 0 \\ \vdots \\  \\ 0 \end{smallmatrix}\right )$.
Then
\begin{equation}\label{eq:main} \ms^\mu(t) =   \frac{\det(\mathrm{M}^\mu)}{\det(\mathrm{I}-t\mathrm{A})} =
 \frac{\det(\mathrm{M}^\mu)}{\prod_{g \in \Gamma} \left(1- \chi_\VV(g)t\right)}, \end{equation}
where $\Gamma$ is a set of conjugacy class representatives of $\GG$ and $\chi_\VV(g)$
is the value of the character $\chi_\VV$ of $\VV$ on $g$.
\end{thm}

\begin{remark}  The space of $\GG$-invariants in $\VV^{\ot k}$ is the sum of the
copies of the trivial $\GG$-module $\GG^{\mathbf 0}$ in $\VV^{\ot k}$.   Hence, the dimension
of the space of $\GG$-invariants in $\VV^{\ot k}$ is  $\ms_k^{\mathbf 0}$, and the Poincar\'e series
$\ms^{\mathbf 0}(t)$ is the generating function for those dimensions.   Under the assumptions of
Theorem \ref{T:Poin}, it follows that
\begin{equation}\label{eq:inv}\ms^{\mathbf 0}(t) =  \frac{\det \left(\mathrm{I} - t \angstrom \right)}{\det \left(\mathrm{I} - t \mathrm{A}\right)}
=  \frac{\det \left(\mathrm{I} - t  \angstrom \right)}{\prod_{g \in \Gamma} \left(1- \chi_\VV(g)t\right)},
\end{equation}
where $\mathrm{A}$ is the adjacency matrix of the representation graph $\mathcal{R}_{\VV}(\GG)$
and  \ $\angstrom$ is the adjacency matrix of the graph obtained from $\mathcal{R}_{\VV}(\GG)$
 by removing the node ${\mathbf 0}$ and all its incident edges.  \end{remark}

Consideration of the minimum polynomial of the adjacency matrix of a graph with finitely many vertices leads to the next result.
\begin{prop}\label{P:recur}  Suppose
$$p(t) = t^d + p_{d-1} t^{d-1} + \cdots + p_1 t + p_0 \in \mathbb C[t]$$
is the minimum polynomial of the adjacency matrix $\mathrm A$ of a finite graph $\mathcal G$,
and   $\ms^\mu_{\nu,k}$ is the number of walks on $\mathcal G$ of $k$ steps  from a node $\nu$
to a node $\mu$.   
\begin{itemize}
\item[{\rm (i)}] The following recursion relation holds for all $k \geq 0$:
\begin{equation}\label{eq:recursion} \ms_{\nu,k+d}^\mu + p_{d-1} \ms_{\nu,k+d-1}^\mu + \cdots + p_1 \ms_{\nu,k+1}^\mu + p_0 \ms_{\nu,k}^\mu = 0.\end{equation}
\item[{\rm (ii)}]  The corresponding exponential generating function,  $\gr^\mu_\nu(t) = \displaystyle{ \sum_{k \geq 0}\ms_{\nu,k}^\mu
 \frac{t^k}{k\,!}}$,
 satisfies the differential equation
\begin{equation}\label{eq:diffeq} y^{(d)} + p_{d-1}y^{(d-1)} + \cdots p_1 y^{(1)} + p_0 y = 0.\end{equation}
\end{itemize}
\end{prop}
\begin{proof}
If $p(t)$ is as above, then
\begin{equation}\label{eq:kmin}  \Ar^{k+d} + p_{d-1}\Ar^{k+d-1} + \cdots + p_1 \Ar^{k+1} + p_0 \Ar^k = \Ar^k p(\Ar) = 0
\end{equation}
for all $k \geq 0$.    Since $(\Ar^\ell)_{\nu, \mu} = \ms_{\nu,\ell}^\mu$ for all $\ell \geq 0$,
taking the $(\nu,\mu)$ entry of \eqref{eq:kmin}  gives the desired result in \eqref{eq:recursion}.
It follows from  \eqref{eq:recursion} that  $\gr_\nu^\mu(t)$ satisfies
  \eqref{eq:diffeq},  where $y^{(r)} = \left(\frac{d}{dt}\right)^r(y)$ for all $r \geq 0$.     \end{proof}
  \smallskip
 
\begin{thm}\label{T:expgen}  \ Let $\GG$ be a finite group with irreducible modules $\GG^\lam$, $\lambda \in \Lambda(\GG)$,  over $\CC$, and let $\VV$ be a faithful finite-dimensional $\GG$-module.
Assume $\gr^\mu(t) = \displaystyle{ \sum_{k \geq 0}\ms_k^\mu\, \frac{t^k}{k\,!}}$ is the exponential generating function for the multiplicity $\ms_k^\mu$  of
$\GG^\mu$ in  $\VV^{\ot k}$ for $k \geq 0$ (equivalently for the number of walks of $k$ steps on $\mathcal R_\VV(\GG)$ from
$\mathbf 0$ to $\mu$).  Then
\begin{itemize}
\item[{\rm (i)}]  $\gr^\mu(t)$ satisfies the differential equation
$y^{(d)} + p_{d-1}y^{(d-1)} + \cdots + p_1 y^{(1)} + p_0 y = 0.$
where
$p(t) = t^d + p_{d-1} t^{d-1} + \cdots + p_1 t + p_0$ is the minimum polynomial of the adjacency matrix
$\mathrm{A}$  of the
representation graph $\mathcal R_{\VV}(\GG)$.
\item[{\rm (ii)}]  The roots of $p(t)$ are  the distinct character values
in $\{\chi_{\VV}(g)  \mid g \in \Gamma\}$,  where $\Gamma$ is a set of conjugacy class representatives
of $\GG$.   When  $p(t)$ has distinct roots  $\xi_j$, $j=1,\dots, d$,
then $\gr^\mu(t)$ is a linear combination of the exponential functions $\er^{\xi_j t}$.
 \end{itemize}
\end{thm}

\begin{proof}  The assertion in (i)  is an  immediate consequence of  Proposition \ref{P:recur}.  That the roots of $p(t)$ are given by character
values follows as in \eqref{eq:main} (compare \cite[Sec.~1]{St}).    \end{proof}

In this work,  we focus on the abelian group  $\GG = \mathbb Z_2^n$,  where
 $\ZZ_2$ denotes the integers $\modd 2$.   This group
appears in many different settings including important ones in computing,  where the
elements $\arm =  (\arm_1, \dots, \arm_n)$, $\arm_i \in
\{0,1\}$ for all $i \in [1,n] := \{1,2, \dots, n\}$,  of $\ZZ_2^n$  are regarded as $n$ bits.
Here it is convenient to think of $\ZZ_2^n$ as
a multiplicative group and write  $\es^\arm$ rather than $\arm$,  so that the group operation is
given by $\es^\arm \es^{\br}  = \es^{\arm+\br}$, $\arm, \br \in \ZZ_2^n$, where the
sum $\arm+\br$ is componentwise addition in $\ZZ_2$.
Since $\ZZ_2^n$ is abelian,   the irreducible $\ZZ_2^n$-modules are all one-dimensional,
and we label them with the elements of $\ZZ_2^n$.     Thus, for $\br  \in \ZZ_2^n$,   let  $\XX^\br= \CC x_\br$,  where
$$ \es^\arm  x_\br  = (-1)^{\arm \cdot \br} x_\br,$$
and $\arm \cdot \br$ is the usual dot product.     Observe that $$\es^{\arm+\arm'}x_\br = (-1)^{(\arm+\arm')\cdot \br} x_\br = (-1)^{\arm \cdot \br} (-1)^{\arm' \cdot \br} x_\br
= \es^\arm(\es^{\arm'} x_\br),$$
so this does in fact define a $\ZZ_2^n$-module action on $\XX^\br$, and the corresponding
 character $\chi_\br$ of $\XX^{\br}$ is given by
 \begin{equation} \chi_\br(\arm) = \mathsf{tr}_{\XX^{\br}}(\es^{\arm}) = (-1)^{\arm \cdot \br}.
 \end{equation}
Moreover, since
$$ \er^\arm (x_\br \ot x_{\cm}) =  (-1)^{\arm \cdot \br} (-1)^{\arm \cdot \cm} x_\br \ot x_\cm = (-1)^{\arm \cdot (\br+\cm)}   x_\br \ot x_\cm,$$
we have that
\begin{equation}\label{eq:tens} \XX^\br \ot \XX^\cm \cong  \XX^{\br+\cm}\end{equation}
for all $\br, \cm \in \ZZ_2^n$.

For $i \in [1,n]$,  let $\ve_i$ denote the $n$-tuple in $\ZZ_2^n$ with $1$ as its $i$th component  and $0$ for all other components.
Set
$$\VV = \XX^{\ve_1} \oplus \XX^{\ve_2} \oplus \cdots \oplus \XX^{\ve_n}.$$
The nodes of the representation graph  $\mathcal R_{\VV}(\ZZ_2^n)$ are just the elements of $\ZZ_2^n$, hence,
are the  vertices  $\arm = (\arm_1, \dots, \arm_n)$,  \ $\arm_i \in \{0,1\}$ for all $i$,
of the $n$-cube (hypercube).  By \eqref{eq:tens},
$\XX^\arm\ot \VV = \sum_{i = 1}^n  \XX^{\arm + \ve_i},$ so that  tensoring $\XX^{\arm}$
with $\VV^{\ot k}$ amounts to taking  a walk of $k$ steps on the $n$-cube
starting from node $\arm$.

We apply Schur-Weyl duality results  to  relate walks  of $k$ steps on  the $n$-cube to the centralizer algebra
 $\ZsZ= \End_{\ZZ_2^n}(\VV^{\ot k})$  and  its irreducible modules $\Zs_k^{\arm}$.
That  connection enables us to give an expression for the  dimension of $\ZsZ$
 and for the dimension of  the module $\Zs_k^{\arm}$  for all $\arm \in \ZZ_2^n$.
 The group $\ZZ_2^n$ is a normal subgroup of the  reflection group
 $\GG(2,1,n) \cong \ZZ_2 \wr \Sr_n$ (the Weyl group of type $\mathrm{B}_n$).
 We identify $\VV$ with the irreducible $\GG(2,1,n)$-module on which elements
 of $\GG(2,1,n)$ act as $n \times n$ signed permutation matrices relative to the
 basis $x_i = x_{\ve_i}$, $i \in [1,n]$.      Tanabe \cite{T} has described a basis for the centralizer algebra $\ZsG = \End_{\GG(2,1,n)}(\VV^{\ot k})$  in terms of diagrams corresponding
 to set partitions.    Since  $\ZZ_2^n
 \subseteq \GG(2,1,n)$,
there is a reverse inclusion of centralizers
 $$\ZsG \subseteq  \ZsZ.$$
 We use that relationship  to index a basis for $\ZsZ$
 by labeled partition diagrams and to  give a formula  for $\dimm \ZsZ$ and for  $\dimm \Zs_k^{\arm}$ by counting
 certain  partitions.
Theorem \ref{T:Poin} can be used to obtain the Poincar\'e series (generating function)
for the dimensions of the irreducible modules
$\Zs_k^{\arm}$, $k \geq 0$ (hence, for the number of  walks of $k$ steps from $\mathbf{0} =(0,\dots,0)$ to $\arm$  on
the $n$-cube).   In the final section, we discuss these series
and also determine the exponential generating functions for these dimensions. The main result of that
section is Theorem \ref{T:expogen},  which says that for $\arm \in \ZZ_2^n$,
$$\gr^{\arm} (t) ={\sum_{k \geq 0}\ms_{k}^{\arm} \frac{t^{k}}{k!}} =\left( \cosh\,t\right)^{n-h}\left(\sinh\,t \right)^{h}$$
where $h = h(\arm)$, the Hamming weight (the number of ones) of $\arm$, and $\cosh\,t$ and $\sinh\,t$ are
hyperbolic cosine and sine.

 The group $\GG(2,1,n)$ also contains the symmetric group $\Sr_n$ as a subgroup,
 and there is  a reverse inclusion of centralizer algebras
 $\ZsG \subseteq  \ZsS,$ which has
 provided a number of interesting results and motivated the study of the
 corresponding Hecke algebras (see for example, \cite{A1,A2,AK,HR1}).

\begin{section}{Walks on graphs and on the $n$-cube} \end{section}
Assume $\Ar$ is the adjacency matrix of a finite graph $\mathcal G$  with undirected edges so that $\Ar$ is a real symmetric matrix.
Let $\mathcal V$ be the vertex set of $\mathcal G$.     Then $\Ar$ has real eigenvalues $\lambda_v, \ v \in \mathcal V$.
Let $\cE_v = (\cE_{u,v})_{u \in \mathcal V}$ be the orthonormal eigenvectors of $\Ar$  so that $\Ar \cE_v = \lambda_v \cE_v$,
and let $\cE = (\cE_{u,v})$ be the matrix whose $v$th column  is the column vector $\cE_v$.
Then the transpose  $\cE^{\tt t} = \cE^{-1}$,  and $\cE \cE^{\tt t} = \cE^{\tt t} \cE = \mathrm{I}_{|\mathcal V|}$.

The next results are from \cite[Chaps. 1,2]{S}. We include the
proofs,  in part  to establish our notation.     \medskip

\begin{prop}\label{P:graph}  The number of walks on the graph $\mathcal {G}$ of $k$-steps from node $v$ to node  $w$ is
$(\Ar^k)_{v,w} =  \sum_{u \in \mathcal V} \cE_{v,u} \cE_{w,u}  \lambda_u^k$.
\end{prop}

\begin{proof}   We have \, $\cE^{-1} \Ar \cE  =  \diag\{\lambda_u\}_{u \in \mathcal V}$, \, so that
$\cE^{-1}\Ar^k \cE = (\cE^{-1} \Ar \cE)^k    =$   $\diag\{\lambda_u^k\}_{u \in \mathcal V}$.      Therefore,
$\Ar^k  =  \cE \, \diag\{\lambda_u^k\} \, \cE^{\tt {t}}$,    and \medskip

$\hspace{1.5cm}  (\Ar^k)_{v,w}  =  \sum_{u \in \mathcal{V}  } \,  \cE_{v,u} \lambda_u^k (\cE^{\tt {t}})_{u,w}
 =   \sum_{ u \in \mathcal {V}} \, \cE_{v,u} \cE_{w,u}\lambda_u^k. $  \end{proof}
 \medskip

We specialize now to the case that $\GG$ is the group $\ZZ_2^n$ and
set  $\mathbf 0 = (0,\ldots, 0)$.   For $\arm \in \ZZ_2^n$,  let $h(\arm)$ be the
{\it Hamming weight} of $\arm$, i.e., the number of ones in $\arm$.
We consider the case that the graph $\mathcal {G}$ is the  representation graph $\mathcal{R}_{\VV_{\mathsf{S}}}(\ZZ_2^n)$ obtained from  the $\ZZ_2^n$-module $\VV_{\mathsf S} = \bigoplus_{\mathrm{s} \in \mathsf{S}}  \XX^{\mathrm{s}}$,   where  $\mathsf{S}$   is a  nonempty  subset of $\ZZ_2^n$.    The next result
gives the eigenvalues and corresponding  eigenvectors of the adjacency matrix $\Ar_{\mathsf{S}}$  for   $\mathcal{R}_{\VV_{\mathsf{S}}}(\ZZ_2^n)$.
Note  that $\XX^\arm \ot \VV_{\mathsf{S}} = \bigoplus_{\mathrm{s} \in \mathsf{S}}  \XX^{\arm + \mathrm{s}}$.       For $\br \in \ZZ_2^n$,  we will write $\underline \br$
for  $2^n \times 1$ column vector with $1$ in the row corresponding to $\br$ and
0 for all its other components.  The argument in \cite[Chap.~2]{S} involves the discrete Radon transform
on $\ZZ_2^n$  (see also
\cite{DG1}), which we don't use here. 
 
\begin{prop}\label{P:cube}  Let $\mathsf{S}$ be a nonempty subset of $\ZZ_2^n$,  and let
$\Ar_{\mathsf{S}}$ be the adjacency matrix of the representation graph $\mathcal{R}_{\VV_{\mathsf{S}}}(\ZZ_2^n)$, where $\VV_{\mathsf S} = \bigoplus_{\mathrm{s} \in \mathsf{S}}  \XX^{\mathrm{s}}$.  Then the character values
$\chi_{{}_{\VV_{\mathsf{S}}}}(\arm): = \chi_{{}_{\VV_{\mathsf{S}}}}(\es^{\arm})  = \sum_{\mathrm{s} \in \mathsf{S}}  (-1)^{\arm \cdot \mathrm{s}}$
for $\arm \in \ZZ_2^n$ are the eigenvalues of $\Ar_{\mathsf{S}}$, and the vector   $\cE_\arm = \sum_{\br \in \ZZ_2^n}  (-1)^{\arm \cdot \br}  \underline \br$ is an eigenvector for $\Ar_{\mathsf {S}}$ corresponding to the eigenvalue $\chi_{{}_{\VV_{\mathsf{S}}}}(\arm)$.
The vectors $\cE_{\arm}, \arm \in \ZZ_2^n$, give a basis for $\CC^{2^n}$.
\end{prop}

\begin{proof}   Observe that

\begin{align*} \hspace{1.8cm}  \Ar_{\mathsf{S}}  \cE_\arm  &= \sum_{\br \in \ZZ_2^n}  (-1)^{\arm \cdot \br}  \left( \sum_{\mathrm s \in \mathsf{S}}  \underline{\br + \mathrm{s}} \right) 
 = \sum_{\cm \in \ZZ_2^n}  \left( \sum_{\mathrm s \in \mathsf{S}}    (-1)^{\arm \cdot( \cm +\mathrm{s})}\right)
\underline{\cm}   \\
&=  \left(\sum_{\mathrm s \in \mathsf{S}} (-1)^{\arm \cdot \mathrm{s}}\right) \left(\sum_{\cm \in  \mathbb Z_2^n}   (-1)^{\arm \cdot \cm} \underline \cm \right) \\
&= \chi_{{}_{\VV_{\mathsf{S}}}}(\arm) \cE_\arm, \end{align*}
so that $\cE_{\arm}$ is an eigenvector corresponding to the eigenvalue  $\chi_{{}_{\VV_{\mathsf{S}}}}(\arm)$.

We view  $\cE_\arm$
as a column vector whose $\br$th component is   $(-1)^{\arm \cdot \br}$.    Taking the
inner product of two such column vectors gives

\begin{align}\label{eq:evorth}  \cE_{\arm} \cdot \cE_{\arm'} & =  \sum_{\br \in \ZZ_2^n}  (-1)^{\arm \cdot \br} (-1)^{\arm' \cdot \br} =   \sum_{\br \in \ZZ_2^n} (-1)^{(\arm + \arm')\cdot \br} \\
& = \begin{cases}  0   & \text{if   $\arm \neq \arm'$    (equivalently, if  $\arm +\arm' \neq {\mathbf 0}$), }\\
2^n  & \text{if   $\arm =  \arm' $    (equivalently, if \ $\arm +\arm' = \mathbf 0$), }
\end{cases} \nonumber
\end{align}
which is just a statement of the well-known  orthogonality of the characters $\chi_{\arm}$ and $\chi_{\arm'}$
for $\arm \neq \arm'$.   Thus, the eigenvectors $\cE_{\arm}$ are orthogonal, hence linearly independent,  and there are $2^n$
of them, so they determine a basis for $\CC^{2^n}$.  \end{proof}

In the special case $\mathsf{S} = \{\ve_i \mid i \in [1,n]\}$, we have
$\VV_{\mathsf{S}} = \VV$,  and  the  representation graph is just the $n$-cube. In this case,
Proposition \ref{P:cube}  gives

\begin{cor}\label{C:cube1}  Let $\Ar$ be the adjacency matrix of the $n$-cube.      Then $\Ar$ has eigenvalues $\chi_\VV(\arm) =\sum_{i=1}^n (-1)^{\arm \cdot \ve_i}  = \sum_{i=1}^n (-1)^{\arm_i} = n-2h(\arm)$ for $\arm = (\arm_1,\dots, \arm_n)  \in  \ZZ_{2}^n$, where $h(\arm)$ is the {\it Hamming weight} of $\arm$,  and  $\cE_\arm = \sum_{\br \in \ZZ_2^n}  (-1)^{\arm \cdot \br}  \underline \br$ is an eigenvector for $\Ar$ corresponding to the eigenvalue $\chi_\VV(\arm)$.  Thus, the eigenvalues of $\Ar$ are $n-2h$  for $h = 0,1,\dots, n$ and $n-2h$ occurs
with multiplicity $\binom{n}{h}$, and the eigenvectors $\cE_{\arm}$, $\arm \in \ZZ_2^n$,  form a basis for $\mathbb C^{2^n}$.
\end{cor}

Our next goal  is to show the following (compare \cite[Cor.~2.5]{S}).  \medskip

\begin{cor}\label{C:cube2} Let $\br, \cm \in \ZZ_2^n$ and suppose $h(\br + \cm) =h$  (i.e. $\br$ and $\cm$ disagree in exactly $h$ coordinates).
Then the number of walks from $\br$ to $\cm$  of $k$ steps on the $n$-cube is given by

$$
(\Ar^k)_{\br,\cm} = \frac{1}{2^n}  \sum_{i=0}^n \sum_{j=0}^h (-1)^j {\binom{h}{j}} {\binom{n-h}{i-j}} (n-2i)^k.
$$  \end{cor}

\begin{proof}   For $\arm \in \ZZ_2^n$, we see from the calculation in \eqref{eq:evorth} that  $\cE_\arm \cdot \cE_{\arm} = 2^n$.
Therefore,  to get the matrix $\cE$ in Proposition \ref{P:graph}, we need to divide $\cE_\arm$  by $2^{n/2}$.
Let $\cE$ be the $(2^n \times 2^n)$-matrix whose $\arm$th column is $2^{-n/2} \cE_\arm$.      Then by Proposition \ref{P:graph},  we have for the adjacency matrix $\Ar$  of the $n$-cube,

$$(\Ar^k)_{\br,\cm} =  2^{-n}\sum_{\arm \in \ZZ_2^n}\cE_{\br,\arm}\cE_{\cm,\arm} \lambda_\arm^k$$ where
$$\cE_{\br,\arm}\cE_{\cm,\arm}\lambda_\arm^k   = (-1)^{\arm \cdot \br} (-1)^{\arm \cdot \cm} \left(\sum_{i=1}^n (-1)^{\arm_i} \right)^k
= (-1)^{\arm \cdot (\br+\cm)} \left(n-2h(\arm)\right)^k,$$
and  $h(\arm)$ is the Hamming weight of $\arm$.

We count the number of $\arm \in \mathbb Z_2^n$ with Hamming weight $h(\arm) = i$ and with $\arm$ having   $j$  ones  in common with $\br+\cm$ for
$j = 0,1, \dots, h = h(\br+\cm)$.
We can choose $j$ ones in $\br+\cm$ that agree with $j$  ones in $\arm$ in ${\binom{h}{j}}$ ways.
The remaining $i-j$ ones in $\arm$ can be placed  in the remaining $n-h$ positions of $\arm$ in ${\binom{n-h}{i-j}}$ ways.
Since $\arm \cdot (\br +\cm) \equiv j \modd 2$,   we have the desired  result.     \end{proof}

\begin{section} {Consequences for the centralizer algebras $\ZsZ$}  \end{section}

 Recall that in  the $n$-cube case $\GG = \ZZ_2^n$, \  $\VV = \XX^{\ve_1} \oplus \cdots \oplus \XX^{\ve_n}$, and  the irreducible modules for $ \ZZ_2^n$ and for its centralizer algebras
 $\ZsZ= \End_{\ZZ_2^n}(\VV^{\ot k})$ are labeled  by elements $\arm \in \ZZ_2^n$.     Then
 Proposition \ref{P:cube} and Corollary \ref{C:cube2} imply    \medskip

\begin{cor}\label{C:cent}  \hspace{-.5cm} \begin{itemize}
\item[{\rm (i)}]  The dimension of the irreducible $\ZsZ$-module $\Zs_k^\arm$ labeled by $\arm \in \ZZ_2^n$ is given by
 $$\dimm \Zs_k^\arm  = (\Ar^k)_{\mathbf{0},\arm} =  \frac{1}{2^n}  \sum_{i=0}^n  \sum_{j=0}^h  (-1)^j {\binom{h}{j}} {\binom{n-h}{i-j}} (n-2i)^k,$$  where $h = h(\arm)$ is the Hamming weight of $\arm$.
 In particular, the irreducible $\ZsZ$-modules labeled by $\arm$ and $\br$ with $h(\arm) =h(\br)$ have the same dimension.
 \item[{\rm (ii)}]    $ \dimm \ZsZ =  \displaystyle{\frac{1}{2^n} \sum_{i=0}^n  {\binom{n}{i}}  (n-2i)^{2k}}$
 \end{itemize}
  \end{cor}
 \medskip
 \begin{remark}  Part  {\rm (ii)} is a special case of {\rm (i)}, since we know by Schur-Weyl duality that
\begin{align*}
\dimm \ZsZ & = \dimm \Zs_{2k}^\mathbf{0}  \\
& = \displaystyle{\frac{1}{2^n}  \sum_{i=0}^n  \sum_{j=0}^h  (-1)^j {\binom{h}{j}} {\binom{n-h}{i-j}} (n-2i)^{2k}}, \ \, \text{where $h = h(\mathbf{0}) = 0$,} \\
& =  \displaystyle{\frac{1}{2^n} \sum_{i=0}^n  {\binom{n}{i}}  (n-2i)^{2k}}. \end{align*}   \end{remark}

Next, we construct an explicit basis for $\ZsZ$.   Let $\{x_i =   x_{\ve_i} \mid i \in [1,n]\}$ be the basis  for
$\VV = \XX^{\ve_1} \oplus \cdots \oplus \XX^{\ve_n}$
such that  $\er^\arm x_i = (-1)^{\arm \cdot \ve_i}x_i = (-1)^{\arm_i}x_i$.
 Then for $\beta = (\beta_1, \dots, \beta_k) \in [1,n]^k$, set
$x_{\beta} = x_{\beta_1} \ot  \cdots \ot x_{\beta_k}$.      The elements
$x_{ \beta}, \  \beta \in [1,n]^k$,  form a basis for $\VV^{\ot k}$ with
$$\es^\arm x_{ \beta} = (-1)^{\arm \cdot \left(\ve_{\beta_1} + \cdots  + \ve_{\beta_k}\right)} x_{ \beta}.$$

Suppose $\Phi   \in \End(\VV^{\ot k})$,  and for $\alpha \in [1,n]^k$  assume
$$\Phi  x_{\alpha} = \sum_{\beta \in [1,n]^k}  \Phi_{\alpha}^{ \beta}  x_{ \beta},$$
where $\Phi_{\alpha}^{ \beta} \in \CC$ for $\alpha,\beta \in [1,n]^k$.
Then  for all $\arm  \in \mathbb Z_2^n$,
\begin{align*}  \es^\arm  \Phi  x_{\alpha} & = \sum_{\beta \in [1,n]^k} (-1)^{\arm \cdot \left(\ve_{\beta_1} + \cdots +\ve_{\beta_k}\right)}   \Phi_{\alpha}^{ \beta}  x_{ \beta}\\
\Phi \es^\arm x_{\alpha} &= (-1)^{\arm \cdot \left(\ve_{\alpha_1} + \cdots + \ve_{\alpha_k}\right)}\sum_{\beta \in [1,n]^k}  \Phi_{\alpha}^{ \beta}  x_{ \beta}.
\end{align*}
Thus, in order for $\Phi$ to belong to $\ZsZ$ we must have $\ve_{\alpha_1} + \cdots + \ve_{\alpha_k} =
\ve_{\beta_1} + \cdots +\ve_{\beta_k}$ for all $\alpha,\beta \in [1,n]^k$ such that $\Phi_{\alpha}^\beta \neq 0$.  Let  $\EE_{\alpha}^\beta \in \End(\VV^{\ot k})$  be given by
\begin{equation} \label{eq:Eab}
\EE_{\alpha}^\beta x_{\gamma} = \delta_{\alpha,\gamma} x_\beta \ \
\text{for all $\gamma \in [1,n]^k$},
\end{equation}
where $\delta_{\alpha,\gamma}$ is the Kronecker delta.
Since the  $ \EE_{\alpha}^\beta$   with  $\ve_{\alpha_1} + \cdots +\ve_{\alpha_k} =
\ve_{\beta_1} + \cdots +\ve_{\beta_k}$  clearly satisfy the requisite condition to belong to $\ZsZ$,  and they span $\ZsZ$, we have the following

\begin{thm}\label{T:base}   A basis for the centralizer algebra $\ZsZ$ is
$$\{ \EE_{\alpha}^\beta  \mid  \alpha,\beta \in [1,n]^k, \  \ve_{\alpha_1} + \cdots +\ve_{\alpha_k} =
\ve_{\beta_1} + \cdots + \ve_{\beta_k}\},$$  where
$\EE_{\alpha}^\beta$  is as in \eqref{eq:Eab}.  \end{thm}

\begin{section}{Partition diagrams} \end{section}

Let $\mathcal P(k,n)$ denote the set partitions of $[1,2k]$ into at most $n$ parts (blocks).  We view the elements of
$\mathcal P(k,n)$ diagrammatically and identify set partitions with their diagrams.
For example, the diagram below corresponds to the set partition
$\{1,4\}$, $\{2,6,8,9\}$, $\{3,10\}$, $\{5,7\}$ in $\mathcal P(5,n)$ for any $n \geq 4$.

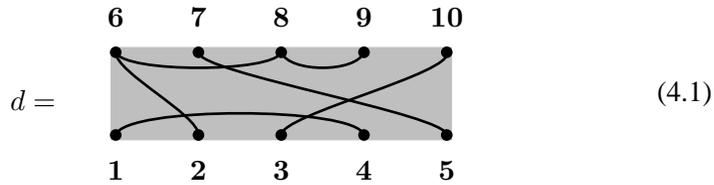
\begin{equation}\label{eq:diag}
\vcenter{\hbox{
\begin{tikzpicture}[scale=1.1,line width=1pt]
\foreach \i in {0,...,5}
{ \path (\i,1) coordinate (T\i); \path (\i,0) coordinate (B\i); }
\filldraw[fill= gray!50,draw=gray!50,line width=4pt]  (T1) -- (T5) -- (B5) -- (B1) -- (T1);
\draw[black] (T1) .. controls +(.1,-.25) and +(-.1,-.25) .. (T3);
\draw[black] (T3) .. controls +(.1,-.25) and +(-.1,-.25) .. (T4);
\draw[black] (T1) .. controls +(0,-.25) and  +(0,.25) .. (B2);
\draw[black] (T5) .. controls +(.1,-.25) and +(0,.25) .. (B3) ;
\draw[black] (T2) .. controls +(0,-.25) and +(0,.25) .. (B5) ;
\draw[black] (B1) .. controls +(.1,.35) and +(-.1,.35) .. (B4) ;
\draw  (B1)  node[black,below=0.2cm]{$\boldsymbol{1}$};
\draw  (B2)  node[black,below=0.2cm]{$\boldsymbol{2}$};
\draw  (B3)  node[black,below=0.2cm]{$\boldsymbol{3}$};
\draw  (B4)  node[black,below=0.2cm]{$\boldsymbol{4}$};
\draw  (B5)  node[black,below=0.2cm]{$\boldsymbol{5}$};
\draw  (T1)  node[black,above=0.2cm]{$\boldsymbol{6}$};
\draw  (T2)  node[black,above=0.2cm]{$\boldsymbol{7}$};
\draw  (T3)  node[black,above=0.2cm]{$\boldsymbol{8}$};
\draw  (T4)  node[black,above=0.2cm]{$\boldsymbol{9}$};
\draw  (T5)  node[black,above=0.2cm]{$\boldsymbol{10}$};
\draw  (T0) node[black,below=0.35cm]{${d = }$};
\foreach \i in {1,...,5}
{ \fill (T\i) circle (2pt); \fill (B\i) circle (2pt); }
\end{tikzpicture}}}
\end{equation}
The way the edges are drawn is immaterial.  What matters is that nodes in the same
block are connected, and nodes in different blocks  are not.

For $d \in \mathcal P(k,n)$,  let
\begin{align*}
& \text{$\mathsf{B}_1$  be the block of $d$ containing 1}; \\
& \text{$\mathsf{B}_2$  be the block containing the smallest number not in $\mathsf{B}_1$};\\
& \quad  \vdots \\
&  \text{$\mathsf{B}_j$ be the block containing the smallest number not in $\mathsf{B}_1 \cup \mathsf{B}_2 \cup \cdots \cup \mathsf{B}_{j-1}$}.
\end{align*}

In the example above, we have ordered the
blocks in this fashion, so that $\mathsf{B}_1 = \{1,4\}$,  $\mathsf{B}_2 = \{2,6,8,9\}$,  $\mathsf{B}_3 = \{3,10\}$, and
$\mathsf{B}_4 = \{5,7\}$.

For $\ell \in [1,2k]$, set
\begin{align}
\qquad  & \zeta_\ell = j   \ \ \text{if  $\ell \in \mathsf{B}_j$, and let} \\
& \zeta_d = (\zeta_1,\dots, \zeta_k) \ \  \text{and }  \zeta_d' = (\zeta_{k+1}, \dots, \zeta_{2k}) \nonumber
\end{align}

In  our running example,
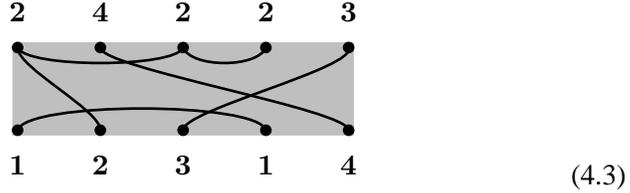
\begin{equation} \label{eq:xid}  {\vspace{-.8cm} \begin{tikzpicture}[scale=1.1,line width=1pt]
\foreach \i in {0,...,5}
{ \path (\i,1) coordinate (T\i); \path (\i,0) coordinate (B\i); }
\filldraw[fill= gray!50,draw=gray!50,line width=4pt]  (T1) -- (T5) -- (B5) -- (B1) -- (T1);
\draw[black] (T1) .. controls +(.1,-.25) and +(-.1,-.25) .. (T3);
\draw[black] (T3) .. controls +(.1,-.25) and +(-.1,-.25) .. (T4);
\draw[black] (T1) .. controls +(0,-.25) and  +(0,.25) .. (B2);
\draw[black] (T5) .. controls +(.1,-.25) and +(0,.25) .. (B3) ;
\draw[black] (T2) .. controls +(0,-.25) and +(0,.25) .. (B5) ;
\draw[black] (B1) .. controls +(.1,.35) and +(-.1,.35) .. (B4) ;
\draw  (B1)  node[black,below=0.2cm]{$\boldsymbol{1}$};
\draw  (B2)  node[black,below=0.2cm]{$\boldsymbol{2}$};
\draw  (B3)  node[black,below=0.2cm]{$\boldsymbol{3}$};
\draw  (B4)  node[black,below=0.2cm]{$\boldsymbol{1}$};
\draw  (B5)  node[black,below=0.2cm]{$\boldsymbol{4}$};
\draw  (T1)  node[black,above=0.2cm]{$\boldsymbol{2}$};
\draw  (T2)  node[black,above=0.2cm]{$\boldsymbol{4}$};
\draw  (T3)  node[black,above=0.2cm]{$\boldsymbol{2}$};
\draw  (T4)  node[black,above=0.2cm]{$\boldsymbol{2}$};
\draw  (T5)  node[black,above=0.2cm]{$\boldsymbol{3}$};
\foreach \i in {1,...,5}
{ \fill (T\i) circle (2pt); \fill (B\i) circle (2pt); }
\end{tikzpicture}}
\end{equation}
so that $\zeta_d = (1,2,3,1,4)$ and $\zeta_d' = (2,4,2,2,3)$.   \bigskip

\begin{subsection}{Definition of $T_d$} \end{subsection}

As in the previous section, assume  $\alpha = (\alpha_1,\dots, \alpha_k) \in [1,n]^k$ and let  $x_\alpha = x_{\alpha_1} \ot \cdots \ot x_{\alpha_k} \in \VV^{\ot k}$, where $x_i = x_{\ve_i}$, $i = 1,\dots, n$.
We can regard $\VV$ as the $n$-dimensional permutation module for the symmetric group
$\Sr_n$, with $\sigma\,x_j  = x_{\sigma(j)}$ for all $j$.   This extends to a diagonal action
of $\Sr_n$ on $\VV^{\ot k}$ with  $\sigma \, x_\alpha = x_{\sigma(\alpha)} =x_{\sigma(\alpha_1)} \ot \cdots \ot x_{\sigma(\alpha_k)}.$

Suppose $T \in \End(\VV^{\ot k})$  and $T  = \sum_{\alpha, \beta \in [1,n]^k}  T_{\alpha}^{\beta}\, \EE_{\alpha}^\beta$,
where the transformations $\EE_{\alpha}^\beta$ are given by \eqref{eq:Eab} and the
$T_{\alpha}^\beta \in \CC$.  Then,
\begin{align}\label{eq:centcond} T \in \End_{\Sr_n}(\VV^{\ot k})   & \iff   \sigma T  = T \sigma  \ \ \text{for all
$\sigma \in \Sr_n$} \nonumber  \\
& \iff  \sum_{\beta \in [1,n]^k}  T_{\alpha}^{\beta} x_{\sigma(\beta)} = \sum_{\beta\in [1,n]^k}  T_{\sigma(\alpha)}^{\beta} x_{\beta}  \ \ \text{for all  $\alpha  \in [1,n]^k$} \nonumber  \\
& \iff   T_{\alpha}^{\beta}  =T_{\sigma(\alpha)}^{\sigma(\beta)}   \ \ \text{for all $\alpha, \beta  \in [1,n]^k$,  $\sigma \in \Sr_n$.}
\end{align}

Now let  $\alpha, \alpha'$ be two $k$-tuples in $[1,n]^k$, but assume $\alpha' = (\alpha_{k+1}, \dots, \alpha_{2k})$ so that the indices on the components of $\alpha'$
run from $k+1$ to $2k$ rather than from $1$ to $k$.  Thus, we can think of $\alpha$  as giving labels for the bottom row
of a partition diagram $d$  and $\alpha'$ as giving labels for the top row of $d$, as pictured below.
$$ 
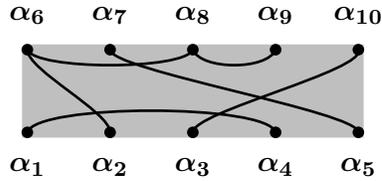
{\begin{tikzpicture}[scale=1.1,line width=1pt]
\foreach \i in {0,...,5}
{ \path (\i,1) coordinate (T\i); \path (\i,0) coordinate (B\i); }
\filldraw[fill= gray!50,draw=gray!50,line width=4pt]  (T1) -- (T5) -- (B5) -- (B1) -- (T1);
\draw[black] (T1) .. controls +(.1,-.25) and +(-.1,-.25) .. (T3);
\draw[black] (T3) .. controls +(.1,-.25) and +(-.1,-.25) .. (T4);
\draw[black] (T1) .. controls +(0,-.25) and  +(0,.25) .. (B2);
\draw[black] (T5) .. controls +(.1,-.25) and +(0,.25) .. (B3) ;
\draw[black] (T2) .. controls +(0,-.25) and +(0,.25) .. (B5) ;
\draw[black] (B1) .. controls +(.1,.35) and +(-.1,.35) .. (B4) ;
\draw  (B1)  node[black,below=0.2cm]{$\boldsymbol{\al_1}$};
\draw  (B2)  node[black,below=0.2cm]{$\boldsymbol{\al_2}$};
\draw  (B3)  node[black,below=0.2cm]{$\boldsymbol{\al_3}$};
\draw  (B4)  node[black,below=0.2cm]{$\boldsymbol{\al_4}$};
\draw  (B5)  node[black,below=0.2cm]{$\boldsymbol{\al_5}$};
\draw  (T1)  node[black,above=0.2cm]{$\boldsymbol{\al_6}$};
\draw  (T2)  node[black,above=0.2cm]{$\boldsymbol{\al_7}$};
\draw  (T3)  node[black,above=0.2cm]{$\boldsymbol{\al_8}$};
\draw  (T4)  node[black,above=0.2cm]{$\boldsymbol{\al_9}$};
\draw  (T5)  node[black,above=0.2cm]{$\boldsymbol{\al_{10}}$};
\foreach \i in {1,...,5}
{ \fill (T\i) circle (2pt); \fill (B\i) circle (2pt); }
\end{tikzpicture}}
$$

Let
\begin{equation}\label{eq:Tddef} T_d  = \sum_{\al, \al' \in [1,n]^k}   \left(T_d\right)_\al^{\al'}  \EE_{\al}^{\al'} \in \End(\VV^{\ot k}), \end{equation}
where
$$ (T_d)_\al^{\al'} = \begin{cases}
1 & \text{if $\al_i = \al_j$ iff $i,j$ are  in the same block in $d$  for  $i,j \in [1,2k]$,}  \\
0 & \text{otherwise}.
\end{cases} $$

In the example above,
$$ (T_d)_\al^{\al'} = \begin{cases} 1 &\text{iff  $\al_1 = \al_4$;  $\al_2 = \al_6 = \al_8 = \al_9$;  $\al_3 = \al_{10}$;  $\al_5 = \al_7$;}   \\
0 &\text{otherwise}.
\end{cases} $$

Observe that $T_d \in \ZsS =  \End_{\Sr_n}(\VV^{\ot k})$,  as $ \left(T_d\right)_\al^{\al'} = 1$ (resp. 0) exactly when $ \left(T_d\right)_{\sigma(\al)}^{\sigma(\al')}  = 1$
(resp. 0) for all $\sigma \in \Sr_n$, so that the condition in \eqref{eq:centcond} is satisfied.    In fact, the transformations $T_d$ as $d$ ranges over the diagrams in $\mathcal P(k,n)$ give
a basis for the centralizer algebra $\ZsS$ (see for example, [HR2]).

For $(\al,  \al')$ and $(\beta,\beta')$, write $(\al, \al') \sim_{\Sr_n} (\beta,\beta')$ if  $\al = \sigma(\beta)$ and $\al'= \sigma(\beta')$ for
some $\sigma \in \Sr_n$.     Then

\begin{equation}\label{eq:Td} T_d  =  \sum_{(\al, \al')\,\sim_{\Sr_n}\,(\zeta_d, \zeta_d')}  \EE_{\al}^{\al'} \\
\end{equation}
and
\begin{equation}\label{eq:Sncent}  \dimm \ZsS = | \mathcal P(k,n)| =
\sum_{j = 1}^{n}  \binombr{2k}{j}
\end{equation}
where $\binombr{2k}{j} $  is the Stirling number
of the 2nd kind,  which counts the number
of ways to partition $2k$ objects into $j$ nonempty blocks.  The Stirling number
 $\binombr{2k}{j} = 0$ whenever $j > 2k$.

Next we describe \, $\ZsG = \End_{\GG(2,1,n)}(\VV^{\ot k})$, \,
where  $\GG(2,1,n) = \ZZ_2 \wr \Sr_n$  (the Weyl group of type $\mathrm B_n$).
 Note that $\GG(2,1,n)$ acts on $\VV$  so that relative to the basis $\{x_i \mid i\in [1,n]\}$,  each element of $\GG(2,1,n)$
acts by a permutation matrix with entries $\pm 1$.  The inclusion  $\Sr_n \subset \GG(2,1,n)$ implies
the reverse
inclusion of centralizer algebras,
$\ZsG \subset  \ZsS$.    In \cite{T}, Tanabe investigated the centralizer algebras of the complex reflection groups $\GG(m,p,n)$ acting
on $\VV^{\ot k}$.      Applying Tanabe's results to the particular case  of $\GG(2,1,n)$,  we have

\medskip

\begin{prop}\label{P:tanabe}  Let $\mathcal P_{\mathsf{even}}(k,n)$ be the set of
partitions of $[1,2k]$ into blocks of even size such that there are at most $n$ blocks.
Then $\{T_d  \mid  d \in \mathcal P_{\mathsf{even}}(k,n)\}$ is a basis for $\ZsG$,
where  $T_d \in \End(\VV^{\ot k})$ is as in \eqref{eq:Tddef} (or equivalently, as in \eqref{eq:Td}).
\end{prop}

\begin{proof}     By \cite[Lem. 2.1]{T},   a basis for  $\ZsG$
consists of the transformations $T_d$ such that $d \in \mathcal P(k,n)$ and
\begin{equation}\label{eq:Tan}
\text{($\#j$ in $\zeta_d$) $\equiv$ ($\#j$ in $\zeta_d'$) $\modd 2$ for each $j \in [1,n]$}.
\end{equation}
This condition is equivalent to saying that  the blocks of $d$ are of even size.   \end{proof}

For the example in \eqref{eq:xid},  $\zeta_d = (1,2,3,1,4)$ and $\zeta_d' = (2,4,2,2,3)$, so  that  ($\#1$ in $\zeta_d)  = 2  \equiv 0  = (\#1$  in  $\zeta_d'$);
($\#2$ in $\zeta_d ) = 1  \equiv 3  = ( \#2$  in  $\zeta_d'$); \  ($\#3$ in $\zeta_d)  = 1  \equiv 1  = (\#3$  in  $\zeta_d'$); and
($\#4$ in $\zeta_d ) = 1  \equiv 1  = (\#4$  in  $\zeta_d'$);  and
($\#j$ in $\zeta_d ) = 0  \equiv 0  = (\#j$  in  $\zeta_d'$) for all $j \in [5,n]$.    There are 4 blocks in $d$,
and they have sizes $2,4,2,2$.   Thus,  $d$ satisfies condition \eqref{eq:Tan},
and $T_d$ is a basis element of $\Zs_5\big(\GG(2,1,n)\big)$.  \medskip

\begin{subsection} {$\dimm \ZsG$ and
$\dimm \ZsZ$} \end{subsection}

Let  $\Tr(k,r)$ be the number of partitions of a set of size $2k$ into $r$ nonempty blocks of
even size. In particular, $\Tr(k,r) = 0$ if $r > k$,  and $\Tr(k,1) = 1$.
These numbers correspond to  sequence A156289 in the Online Encyclopedia of Integer Sequences \cite{OEIS},
and are known to satisfy

\begin{equation}\label{eq:Tkr1}  \Tr(k,r) = \frac{1}{r!\,2^{r-1}}  \sum_{j=1}^r (-1)^{r-j} {\binom{2r}{r-j}} j^{2k}\end{equation}
In particular, $\Tr(4,2) = \frac{1}{4}\left( (-1)^{2-1}{\binom{4}{1}}1^8 + (-1)^0{\binom{4}{0}} 2^8\right) =
\frac{1}{4}(256-4) = 63$.
Each such set partition determines an integer solution to
\begin{equation}\label{eq:intsol}  \lam_1 +  \lam_2 + \cdots + \lam_r = k,    \qquad \lam_1 \geq \lam_2 \geq \dots \geq \lam_r > 0; \end{equation}
hence,  a partition $\lam = (\lam_1,\dots,\lam_r)$  of $k$ into $r$ nonzero parts.   Let
$\ell^\lam_j$ be the multiplicity of $j$ in the partition $\lam$.
Then  an alternate expression for $\Tr(k,r)$ is
\begin{align}\label{eq:Tkr}  \Tr(k,r) &=  \sum \, \frac{1}{\ell^\lam_1!\,\ell^\lam_2! \,\dots}\,\,
{\binom{2k}{2\lam_1 \  2\lam_2 \ \cdots \ 2 \lam_r}}  \quad \text{(multinomial notation)}  \nonumber \\
&=  \sum \, \frac{1}{\ell^\lam_1!\,\ell^\lam_2! \,\dots}\,\,{\binom{2k}{2\lam_1}} {\binom{2k-2\lam_1}{2\lam_2}} \ \cdots \ {\binom{2\lam_r}{2\lam_r}},
 \end{align}
where the sum is over all $\lam = (\lam_1,\dots,\lam_r)$ satisfying \eqref{eq:intsol}.
For example, when $k = 4$ and $r=2$,  there are two solutions $3+1 = 4$, $2 +2 = 4$ to \eqref{eq:intsol}
and
$$\Tr(4,2) = {\binom{8}{6}}{\binom{2}{2}} + \frac{1}{2!}{\binom{8}{4}}{\binom{4}{4}} = 28+35 = 63.$$
\medskip

The next result is an immediate consequence of  Proposition \ref{P:tanabe}.
\medskip
\begin{prop}\label{P:dim1}   $\dimm \ZsG = \sum_{r=1}^n  \Tr(k,r)$.
\end{prop}

Recall that the transformations $\EE_\al^\beta$,  where $\al, \beta \in [1,n]^k$ and $\sum_{i=1}^k  \ve_{\al_i}  = \sum_{i=1}^k \ve_{\beta_i}$,
form a basis for $\ZsZ$.     Note that the condition $\sum_{i=1}^k  \ve_{\al_i}  = \sum_{i=1}^k \ve_{\beta_i}$ says that  ($\#j$ in $\al$) $\equiv$ ($\#j$ in $\beta$) $\modd 2$ for each $j \in [1,n]$.   For example,  if
$\al = (4,3,2,4,1)$ and $\beta = (3,1,3,3,2)$,   then this condition is satisfied.    Moreover, if we label the nodes
of a diagram with the components of $\al$ on the bottom, and the components of $\beta$ on top  and connect
nodes that have the same label, we obtain a diagram $d$ satisfying \eqref{eq:Tan},
since there is a $\sigma \in \Sr_n$ such that $\sigma(\al) = \zeta_d$ and
$\sigma(\beta) = \zeta_d'$ (in fact,  in this example $\sigma = (1\,4)(2 \, 3)$ will do the job).

$$ 
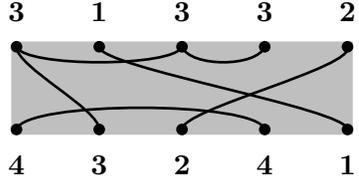
{\begin{tikzpicture}[scale=1.1,line width=1pt]
\foreach \i in {0,...,5}
{ \path (\i,1) coordinate (T\i); \path (\i,0) coordinate (B\i); }
\filldraw[fill= gray!50,draw=gray!50,line width=4pt]  (T1) -- (T5) -- (B5) -- (B1) -- (T1);
\draw[black] (T1) .. controls +(.1,-.25) and +(-.1,-.25) .. (T3);
\draw[black] (T3) .. controls +(.1,-.25) and +(-.1,-.25) .. (T4);
\draw[black] (T1) .. controls +(0,-.25) and  +(0,.25) .. (B2);
\draw[black] (T5) .. controls +(.1,-.25) and +(0,.25) .. (B3) ;
\draw[black] (T2) .. controls +(0,-.25) and +(0,.25) .. (B5) ;
\draw[black] (B1) .. controls +(.1,.35) and +(-.1,.35) .. (B4) ;
\draw  (B1)  node[black,below=0.2cm]{$\boldsymbol{4}$};
\draw  (B2)  node[black,below=0.2cm]{$\boldsymbol{3}$};
\draw  (B3)  node[black,below=0.2cm]{$\boldsymbol{2}$};
\draw  (B4)  node[black,below=0.2cm]{$\boldsymbol{4}$};
\draw  (B5)  node[black,below=0.2cm]{$\boldsymbol{1}$};
\draw  (T1)  node[black,above=0.2cm]{$\boldsymbol{3}$};
\draw  (T2)  node[black,above=0.2cm]{$\boldsymbol{1}$};
\draw  (T3)  node[black,above=0.2cm]{$\boldsymbol{3}$};
\draw  (T4)  node[black,above=0.2cm]{$\boldsymbol{3}$};
\draw  (T5)  node[black,above=0.2cm]{$\boldsymbol{2}$};
\foreach \i in {1,...,5}
{ \fill (T\i) circle (2pt); \fill (B\i) circle (2pt); }
\end{tikzpicture}.}
$$
 Therefore,  $\EE_\al^\beta$ is one of the summands of $T_d$.
We note that $\EE_\al^\beta$ doesn't commute with $\GG(2,1,n)$.   But
it does commute with its normal subgroup $\ZZ_2^n$.  \medskip

\begin{example}  Suppose $n = 3$ and $k = 2$.     The dimension of $\Zs_2(\Sr_3)$ is
$\vert\mathcal P(2,3)\vert = \binombr{4}{1}  + \binombr{4}{2}  +  \binombr{4}{3}  = 14$ where the summands are Stirling numbers
of the 2nd kind.      There are only 4 diagrams $d \in \mathcal P(2,3)$  that
have $\leq 3$ blocks of even size; namely,  the ones pictured below,  where we have indicated $\zeta_d$
and $\zeta_d'$ on each diagram.

$$ 
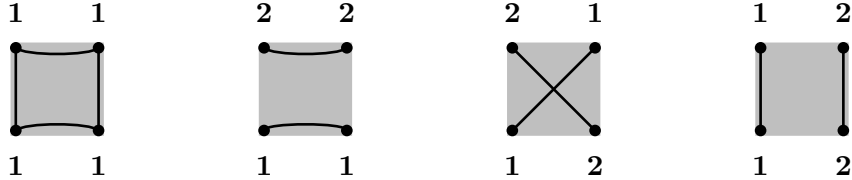
{\begin{tikzpicture}[scale=1.1,line width=1pt]
\foreach \i in {0,...,10}
{ \path (\i,1) coordinate (T\i); \path (\i,0) coordinate (B\i); }
\filldraw[fill= gray!50,draw=gray!50,line width=4pt]  (T0) -- (T1) -- (B1) -- (B0) -- (T0);
\filldraw[fill= gray!50,draw=gray!50,line width=4pt]  (T3) -- (T4) -- (B4) -- (B3) -- (T3);
\filldraw[fill= gray!50,draw=gray!50,line width=4pt]  (T6) -- (T7) -- (B7) -- (B6) -- (T6);
\filldraw[fill= gray!50,draw=gray!50,line width=4pt]  (T9) -- (T10) -- (B10) -- (B9) -- (T9);
\draw[black] (T0) .. controls +(.1,-.10) and +(-.1,-.10) .. (T1);
\draw[black] (B0) .. controls +(.1,.10) and +(-.1,.10) .. (B1);
\draw[black] (T0) .. controls +(0, 0) and +(0,0) .. (B0);
\draw[black] (T1) .. controls +(0, 0) and +(0,0) .. (B1);
\draw[black] (T3) .. controls +(.1,-.10) and +(-.1,-.10) .. (T4);
\draw[black] (B3) .. controls +(0,.10) and  +(0,.10) .. (B4);
\draw[black] (T6) .. controls +(0,0) and +(0,0) .. (B7) ;
\draw[black] (T7) .. controls +(0,0) and +(0,0) .. (B6) ;
\draw[black] (T9) .. controls +(0,0) and +(0,0) .. (B9) ;
\draw[black] (T10) .. controls +(0,0) and +(0,0) .. (B10) ;
\draw  (B0)  node[black,below=0.2cm]{$\boldsymbol{1}$};
\draw  (B1)  node[black,below=0.2cm]{$\boldsymbol{1}$};
\draw  (B3)  node[black,below=0.2cm]{$\boldsymbol{1}$};
\draw  (B4)  node[black,below=0.2cm]{$\boldsymbol{1}$};
\draw  (B6)  node[black,below=0.2cm]{$\boldsymbol{1}$};
\draw  (B7)  node[black,below=0.2cm]{$\boldsymbol{2}$};
\draw  (B9)  node[black,below=0.2cm]{$\boldsymbol{1}$};
\draw  (B10) node[black,below=0.2cm]{$\boldsymbol{2}$};
\draw  (T0)  node[black,above=0.2cm]{$\boldsymbol{1}$};
\draw  (T1)  node[black,above=0.2cm]{$\boldsymbol{1}$};
\draw  (T3)  node[black,above=0.2cm]{$\boldsymbol{2}$};
\draw  (T4)  node[black,above=0.2cm]{$\boldsymbol{2}$};
\draw  (T6)  node[black,above=0.2cm]{$\boldsymbol{2}$};
\draw  (T7)  node[black,above=0.2cm]{$\boldsymbol{1}$};
\draw  (T9)  node[black,above=0.2cm]{$\boldsymbol{1}$};
\draw  (T10)  node[black,above=0.2cm]{$\boldsymbol{2}$};
\foreach \i in {0,1}
{ \fill (T\i) circle (2pt); \fill (B\i) circle (2pt); }
\foreach \i in {3,4}
{ \fill (T\i) circle (2pt); \fill (B\i) circle (2pt); }
\foreach \i in {6,7}
{ \fill (T\i) circle (2pt); \fill (B\i) circle (2pt); }
\foreach \i in {9,10}
{ \fill (T\i) circle (2pt); \fill (B\i) circle (2pt); }
\end{tikzpicture}}
$$
Assuming these diagrams are numbered $d_1, \dots, d_4$ from left to right, we have
\begin{align*} T_{d_1} &= \EE_{11}^{11} + \EE_{22}^{22} + \EE_{33}^{33} \\
T_{d_2} & = \EE_{11}^{22} + \EE_{22}^{33} + \EE_{33}^{11} + \EE_{11}^{33} + \EE_{33}^{22} + \EE_{11}^{22} \\
T_{d_3} & = \EE_{12}^{21} + \EE_{23}^{32} + \EE_{31}^{13} + \EE_{13}^{31} + \EE_{32}^{23} + \EE_{21}^{12} \\
T_{d_4} &= \EE_{12}^{12} + \EE_{23}^{23} + \EE_{31}^{31} + \EE_{13}^{13} + \EE_{32}^{32} + \EE_{21}^{21},
\end{align*}
and $\{T_{d_j} \mid j \in [1,4]\}$ is a basis for $\Zs_2\big(\GG(2,1,3)\big)$.   Note there are a total of  21 summands  $\EE_{\alpha}^\beta$ in these expressions.
According to Corollary \ref{C:cent},
$$\dimm \Zs_2(\ZZ_2^3) =  \frac{1}{2^3} \sum_{i=0}^3 {\binom{3}{i}} (3-2i)^4  =  \frac{1}{8}\left(3^4 + 3 + 3 + 3^4\right) = 21.$$

\end{example}

Recall that a  basis element $\EE_\alpha^\beta$ for $\ZsZ$  corresponds to
a partition diagram $d$ with $2k$ nodes  obtained by
 labeling  the nodes
of $d$  with the components of $\al$ on the bottom, and the components of $\beta$ on top.
Nodes  having the same label are connected by an edge.  The blocks have even
size,  and there are  $r$ blocks for some $r \leq n$.    Labeling the blocks with different numbers  amounts to
coloring  the blocks of $d$ with different colors chosen from $n$ colors.    Therefore,
there are $\displaystyle{\Tr(k,r) \frac{n\,!}{(n-r)!}}$ basis elements  $\EE_\alpha^\beta$  corresponding
to diagrams with $r$ blocks.   Combining this with Schur-Weyl duality  gives

\begin{prop}\label{P:dimZ}  \label{eq:dim2}   \begin{align}
\dimm\ZsZ & = \ \sum_{r=1}^n  \Tr(k,r) \frac{n\,!}{(n-r)!} \\
&= \text{the number  of  walks  of  $2k$  steps}  \nonumber \\
&  \quad \  \text{from  $\mathbf{0}$   to $\mathbf{0}$  on  the   $n$-cube}.\nonumber
\end{align}
\end{prop}

\begin{example}
\begin{align*} \dimm \Zs_2(\ZZ_2^3) &=  \Tr(2,1)\frac{3!}{2!} + \Tr(2,2)\frac{3!}{1!} \\
& = 1 \cdot 3 + 3 \cdot 6  = 21. \end{align*}  \end{example}

For $i \in [1,n]$, let $\ps_i: \VV \rightarrow \XX^{\ve_i}$ be the projection map onto the $i$th summand.
For $\alpha = (\alpha_1, \dots, \alpha_k) \in [1,n]^k$,  set
$$\ps_{\alpha} = \ps_{\alpha_1}  \ot \ps_{\alpha_2}\ot  \cdots \ot \ps_{\alpha_k} \in \End(\VV^{\ot k}).$$   Then  for
$\beta \in [1,n]^k$,  $\ps_{\alpha}(x_\beta)  = \prod_{j=1}^k \delta_{\alpha_j, \beta_j}  x_\alpha
= \delta_{\alpha,\beta} x_\alpha$.     Moreover,
$$\ps_\beta T_d  \ps_{\alpha}  = \EE_{\alpha}^\beta.$$
Combining the results of this section, we have

\begin{prop}\label{P:gens}  The elements $T_d$ such that $d \in \mathcal P_{\mathsf{even}}(k,n)$
together with the projections $\ps_\alpha$, $\alpha \in [1,n]^k$,  generate  the centralizer
algebra $\ZsZ$.  \end{prop}

\begin{subsection}{The Bratteli diagram and a bijection} \end{subsection}

The \emph{Bratteli diagram} $\mathcal{B}_{\VV}(\ZZ_2^n)$ associated to the group  $\ZZ_2^n$ and  the module $\VV$
 is the infinite graph  with vertices labeled by the elements of  $\arm \in \ZZ_2^n$  on level $k$
 that can be reached by  a  walk of $k$ steps on the representation graph $\mathcal{R}_{\VV}(\ZZ_2^n)$. Such a walk corresponds to a sequence $\left(\arm^0,{\arm^1},  \ldots, {\arm^k}\right)$
 starting at $\arm^0 = \mathbf{0}= (0,\ldots,0)$,  such that  $\arm^j \in \ZZ_2^n$   for each $1 \le j \le k$, and  $\arm^{j} = \arm^{j-1} + \ve_i$ for some $i \in [1,n]$.  Thus,
$\arm^{j}$ is connected to $ \arm^{j-1}$   by an edge in $\mathcal{R}_{\VV}(\ZZ_2^n)$.
This corresponds to a unique path on $\mathcal{B}_\VV(\ZZ_2^n)$
starting at  $\mathbf{0}$ on the top  and going to  $\arm$ at level $k$.
The subscript on node  $\arm$ at level $k$  in $\mathcal{B}_\VV(\ZZ_2^n)$  indicates  the number
$\ms_k^\arm$  of such paths (hence, the number of walks on $\mathcal R_\VV(\ZZ_2^n)$ of $k$ steps
from $\mathbf{0}$ to $\arm$). This can be easily  computed by summing, in a Pascal triangle fashion, the subscripts of the vertices at level $k-1$ that are connected to $\arm$.  This is  the multiplicity of
the irreducible $\ZZ_2^n$-module $\XX^\arm$  in $\VV^{\ot k}$, which is also the
dimension of the irreducible $\ZsZ$-module  $\Zs_k^\arm$ by Schur-Weyl duality.  The sum of the squares of
those dimensions at level $k$ is the number on the right, which is the dimension of the centralizer algebra $ \ZsZ$.   Levels 0,1,\dots,6  of the Bratteli diagram for $n=3$ are displayed below, where  to simplify the notation we have omitted the commas in writing the
elements $\arm$ of $\ZZ_2^n$.

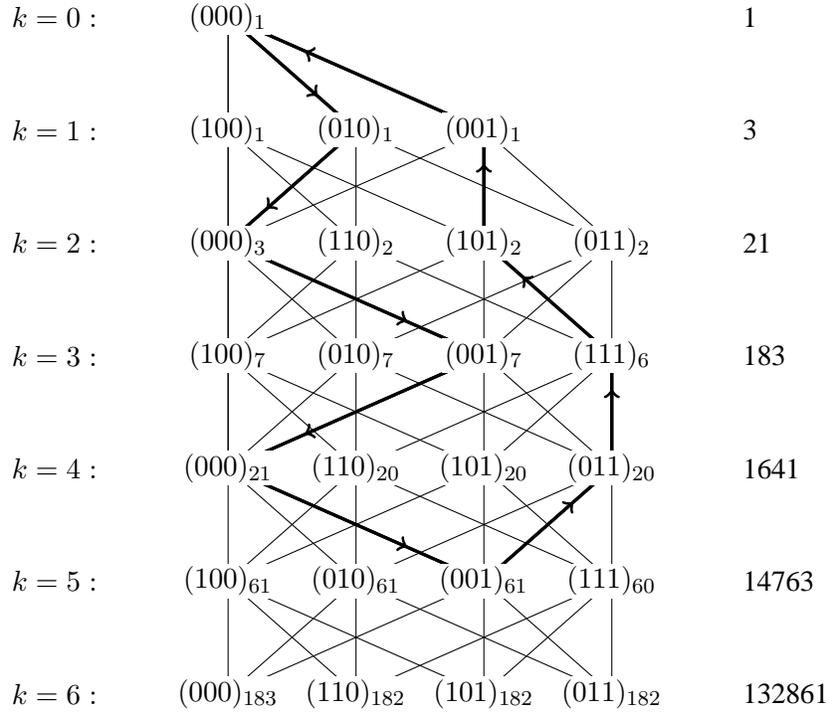
$$
\begin{tikzpicture}[yscale=-1, xscale=1.7]
	\coordinate (00) at (0,0);
	\foreach \x in {0, ..., 2}{\coordinate (1\x) at (\x,1.5);}
	\foreach \x in {0, ..., 3}{\coordinate (2\x) at (\x,3);}
	\foreach \x in {0, ..., 3}{\coordinate (3\x) at (\x,4.5);}
	\foreach \x in {0, ..., 3}{\coordinate (4\x) at (\x,6);}
	\foreach \x in {0, ..., 3}{\coordinate (5\x) at (\x,7.5);}
	\foreach \x in {0, ..., 3}{\coordinate (6\x) at (\x,9);}

    \draw[] (00)--(40);
	\foreach \x in {0, ..., 2}{ \draw[] (00)--(1\x);}
    \draw[] (10)--(20);
    \draw[] (10)--(21);
    \draw[] (10)--(22);

    \draw[] (11)--(20);
    \draw[] (11)--(21);
    \draw[] (11)--(23);
    \draw[] (12)--(20);
    \draw[] (12)--(23);
    \draw[] (12)--(22);

    \draw[] (30)--(40);
    \draw[] (30)--(41);
    \draw[] (30)--(42);
    \draw[] (31)--(40);
    \draw[] (31)--(41);
    \draw[] (31)--(43);
    \draw[] (32)--(40);
    \draw[] (32)--(43);
    \draw[] (32)--(42);
     \draw[] (33)--(41);
    \draw[] (33)--(43);
    \draw[] (33)--(42);

     \draw[] (50)--(60);
    \draw[] (50)--(61);
    \draw[] (50)--(62);
    \draw[] (51)--(60);
    \draw[] (51)--(61);
    \draw[] (51)--(63);
    \draw[] (52)--(60);
    \draw[] (52)--(63);
    \draw[] (52)--(62);
     \draw[] (53)--(61);
    \draw[] (53)--(63);
    \draw[] (53)--(62);

    \foreach \x in {0, ..., 2}{ \draw[] (20)--(3\x);}
    \draw[] (21)--(30);
    \draw[] (21)--(31);
    \draw[] (21)--(33);

    \draw[] (22)--(30);
    \draw[] (22)--(32);
    \draw[] (22)--(33);

    \draw[] (23)--(31);
    \draw[] (23)--(32);
    \draw[] (23)--(33);

     \foreach \x in {0, ..., 2}{ \draw[] (40)--(5\x);}
    \draw[] (41)--(50);
    \draw[] (41)--(51);
    \draw[] (41)--(53);

    \draw[] (42)--(50);
    \draw[] (42)--(52);
    \draw[] (42)--(53);

    \draw[] (43)--(51);
    \draw[] (43)--(52);
    \draw[] (43)--(53);

\begin{scope}[very thick]
    \draw[] (00)--(11);
    \draw[] (11)--(20);
   \draw[] (20)--(32);
   \draw[] (32)--(40);
    \draw[] (40)--(52);
    \draw[] (52)--(43);
    \draw[] (43)--(33);
    \draw[] (33)--(22);
     \draw[] (22)--(12);
     \draw[] (12)--(00);
       \end{scope}
\begin{scope}[very thick,decoration={markings,mark=at position 0.7 with {\arrow{>}}}]
     \draw[postaction={decorate}] (00)--(11);
     \draw[postaction={decorate}] (20)--(32);
     \draw[postaction={decorate}] (11)--(20);
            \draw[postaction={decorate}](32)--(40);
                      \draw[postaction={decorate}](40)--(52);
 \draw[postaction={decorate}](52)--(43);
  \draw[postaction={decorate}](43)--(33);
   \draw[postaction={decorate}](33)--(22);
   \draw[postaction={decorate}](22)--(12);
    \draw[postaction={decorate}](12)--(00);
\end{scope}

\begin{scope}[every node/.style={fill=white, inner sep = 1pt}]
	\node at (00) {$(000)_{1}$};
	\node[left] at (-1, 0) { $k=0:$};
	\node[right] at (4, 0) {1};
    \node at (10) {$(100)_{1}$};
    \node at (11) {$(010)_1$};
    \node at (12) {$(001)_1$};
	\node[left] at (-1, 1.5) { $k=1:$};
	\node[right] at (4, 1.5) {3};

    \node at (20) {$(000)_3$};
    \node at (21) {$(110)_2$};
    \node at (22) {$(101)_2$};
    \node at (23) {$(011)_2$};

    \node[left] at (-1, 3) { $k=2:$};
    \node[right] at (4, 3) {21};

    \node at (30) {$(100)_7$};
    \node at (31) {$(010)_7$};
    \node at (32) {$(001)_7$};
    \node at (33) {$(111)_6$};

    \node[left] at (-1, 4.5) { $k=3:$};
    \node[right] at (4, 4.5) {183};

     \node at (40) {$(000)_{21}$};
    \node at (41) {$(110)_{20}$};
    \node at (42) {$(101)_{20}$};
    \node at (43) {$(011)_{20}$};

     \node[left] at (-1, 6) { $k=4:$};
    \node[right] at (4, 6) {1641};

     \node at (50) {$(100)_{61}$};
    \node at (51) {$(010)_{61}$};
    \node at (52) {$(001)_{61}$};
    \node at (53) {$(111)_{60}$};

    \node[left] at (-1, 7.5) { $k=5:$};
    \node[right] at (4, 7.5) {14763};

     \node at (60) {$(000)_{183}$};
    \node at (61) {$(110)_{182}$};
    \node at (62) {$(101)_{182}$};
    \node at (63) {$(011)_{182}$};

     \node[left] at (-1, 9) { $k=6:$};
    \node[right] at (4, 9) {132861};
\end{scope}
\end{tikzpicture}
$$

A pair  $(\varrho_1,\varrho_2)$  of paths $\varrho_1 = (\arm^0, \arm^1,\dots, \arm^k)$, \
$\varrho_2= (\br^0, \br^1,\dots, \br^k)$  starting from $\arm^0 = \br^0= \mathbf{0}=(0,\dots,0)$ at the top
and going  to $ \arm^k = \br^k = \arm \in \ZZ_2^n$  at level $k$ in  $\mathcal{B}_\VV(\ZZ_2^n)$
determines a closed path from  $\mathbf{0}$  to $\mathbf{0}$  by reversing the second path and concatenating
the two paths.   This is illustrated by the darkened edges in the diagram  above.
Such a pair  determines
two $k$-tuples  $\alpha = ({\alpha_1}, \dots,{\alpha_k})$, and $\beta = ({\beta_1},
\ldots,{\beta_k})$ in $[1,n]^k$, such that  $\arm^j  = \arm^{j-1} + \ve_{\alpha_j}$ for $j=1,\dots,k$,
and $\br^{j-1} = \br^j + \ve_{\beta_{k+1-j}}$ for $j = k,\dots, 1$.    Then since
$\ve_{\alpha_1} + \cdots +\ve_{\alpha_k} +
\ve_{\beta_1} + \cdots +\ve_{\beta_k} = \mathbf{0}$, the condition in Theorem \ref{T:base} is satisfied, and
 $\alpha$ and $\beta$ determine a
labeled partition diagram $d  \in \mathcal P_{\mathsf{even}}(k,n)$, in which nodes  with the same label are connected by an edge.    \medskip

\begin{para}\label{para:bij}  This process establishes a bijection between the pairs $(\varrho_1, \varrho_2)$
of paths from $\mathbf{0}$ at the top of the Bratteli diagram to
$\arm \in \ZZ_2^n$ at level $k$ and the basis elements $\EE_{\alpha}^\beta$ of
$\ZsZ$ such that  {\rm (a)}  $\alpha,\beta \in [1,n]^k$;   {\rm (b)}  $(\# \alpha_i = j) \equiv
(\#\beta_i = j) \modd 2$ for all $j \in [1,n]$;  and  {\rm (c)} $\sum_{i=1}^k \ve_{\alpha_i} =
\sum_{i=1}^k \ve_{\beta_i} = \arm$.  \end{para}

In the example above,  $\ve_2,\ve_2, \ve_3, \ve_3,\ve_3$ have been added
in succession to (000) to
arrive at $\arm = (001)$ so that $\alpha=(2,2,3,3,3)$;   and
 $\ve_2, \ve_1, \ve_2, \ve_1,\ve_3$ have been added  in succession to $\arm = (001)$
to return back $(000)$ so that $\beta = (2,1,2,1,3)$.    The resulting labeled partition diagram
$d \in \mathcal P_{\mathsf{even}}(5,3)$ is

$$ 
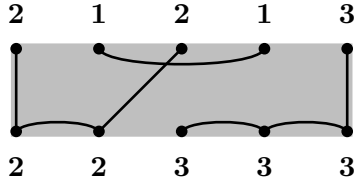
{\begin{tikzpicture}[scale=1.1,line width=1pt]
\foreach \i in {0,...,4}
{ \path (\i,1) coordinate (T\i); \path (\i,0) coordinate (B\i); }
\filldraw[fill= gray!50,draw=gray!50,line width=4pt]  (T0) -- (T4) -- (B4) -- (B0) -- (T0);
\draw[black] (B0) .. controls +(.1,.15) and +(-.1,.15) .. (B1);
\draw[black] (B2) .. controls +(.1,.15) and +(-.1,.15) .. (B3);
\draw[black] (B3) .. controls +(.1,.15) and +(-.1,.15) .. (B4);
\draw[black] (T1) .. controls +(.1,-.25) and +(-.1,-.25) .. (T3);
\draw[black] (T0) .. controls +(0, 0) and +(0,0) .. (B0);
\draw[black] (T2) .. controls +(0, 0) and +(0,0) .. (B1);
\draw[black] (T4) .. controls +(0, 0) and +(0,0) .. (B4);
\draw  (B0)  node[black,below=0.2cm]{$\boldsymbol{2}$};
\draw  (B1)  node[black,below=0.2cm]{$\boldsymbol{2}$};
\draw  (B2)  node[black,below=0.2cm]{$\boldsymbol{3}$};
\draw  (B3)  node[black,below=0.2cm]{$\boldsymbol{3}$};
\draw  (B4)  node[black,below=0.2cm]{$\boldsymbol{3}$};
\draw  (T0)  node[black,above=0.2cm]{$\boldsymbol{2}$};
\draw  (T1)  node[black,above=0.2cm]{$\boldsymbol{1}$};
\draw  (T2)  node[black,above=0.2cm]{$\boldsymbol{2}$};
\draw  (T3)  node[black,above=0.2cm]{$\boldsymbol{1}$};
\draw  (T4)  node[black,above=0.2cm]{$\boldsymbol{3}$};
\foreach \i in {0,...,4}
{ \fill (T\i) circle (2pt); \fill (B\i) circle (2pt); }
\end{tikzpicture}}
$$
which is identified with the
basis element $\EE_{22333}^{21213}$ of $\Zs_5(\ZZ_2^3)$.

\begin{remark} It is evident that the following hold:
\begin{itemize} \item[{\rm (i)}]  For $\arm \in \ZZ_2^n$,  a basis for the $\Zs_k(\ZZ_2^n)$-module
$\Zs_k^\arm \subseteq \VV^{\ot k}$ is
$$
\left \{x_{\alpha} = x_{\alpha_1} \ot \cdots \ot x_{\alpha_k} \, \big | \,
\alpha_j   \in  [1,n] \,  \text{ for all  $j \in \, [1,k]$, and  $\textstyle{ \sum_{j=1}^k} \ve_{\alpha_j} = \arm$} \right \}.$$
 \item[{\rm (ii)}]  $\es^\br \cdot x_\alpha =  (-1)^{\arm \cdot \br} x_\alpha$ for all $\br \in \ZZ_2^n$ and all $x_\alpha$
 in {\rm (i)} so that $\Zs_k^\arm$ is  also a $\ZZ_2^n$-submodule of $\VV^{\ot k}$;  it is the sum of all copies of  the
 $\ZZ_2^n$-module $\XX^\arm$ in $\VV^{\ot k}$.
\item[{\rm (iii)}]   $\End_{\ZZ_2^n}(\Zs_k^\arm)$ has a basis consisting of the transformations
$\EE_{\alpha}^\beta \in \End(\VV^{\ot k})$ such that \,
{\rm (a)}  $\alpha,\beta \in [1,n]^k$;  \,  {\rm (b)} $(\# \alpha_i = j) \equiv
(\#\beta_i = j) \modd 2$ for all $j \in [1,n]$;  and  {\rm (c)} $\sum_{i=1}^k \ve_{\alpha_i} =
\sum_{i=1}^k \ve_{\beta_i} = \arm$.
\end{itemize}
\end{remark}

\begin{subsection}{Poincar\'e series and exponential generating functions} \end{subsection}
The assumptions of Theorem \ref{T:Poin} hold for $\GG = \ZZ_2$ and $\VV = \XX^{\ve_1} \oplus
\cdots \oplus \XX^{\ve_n}$, so by \eqref{eq:main},  the Poincar\'e
series for the multiplicities $\ms_k^{\arm}$ of the irreducible $\ZZ_2^n$-module
labeled by $\arm \in \ZZ_2^n$ in $\TT(\VV) = \bigoplus_{k \geq 0} \VV^{\ot k}$
(hence, for the number $\ms_k^{\arm}$ of walks of $k$ steps from $\textbf{0}$ to $\arm$ on
the $n$-cube) is given by

\begin{equation}\label{eq:PZ}  \ms^{\arm}(t) = \sum_{k \geq 0} \ms_k^{\arm} t^k =   \frac{\det(\mathrm{M}^\arm)}{\det(\mathrm{I}-t\mathrm{A})} =
 \frac{\det(\mathrm{M}^\arm)}{\prod_{\cm \in \ZZ_2^n} \left(1- \chi_\VV(\cm)t\right)}, \end{equation}
where $\Ar$ is the adjacency matrix of the $n$-cube,  and $\mathrm M^\arm$ is the matrix
obtained from $\mathrm{I}-t\mathrm{A}$ by replacing column $\arm$ with $\underline{\delta} = \left (\begin{smallmatrix}  1\\ 0 \\ \vdots \\  \\ 0 \end{smallmatrix}\right )$.

Here we demonstrate how to compute these series.  Since
$\chi_\VV = \sum_{i=1}^n  \chi_{\ve_i}$,
the character values are
$$\chi_{\VV}(\cm)  = \sum_{i=1}^n  \chi_{\ve_i}(\cm) = \sum_{i=1}^n  (-1)^{\cm \cdot \ve_i} =
n - 2h(\cm), $$
where $h(\cm)$ is the Hamming weight of $\cm \in \ZZ_2^n$.
Elements with the same Hamming weight have the same character value on $\VV$.
Therefore the denominator in \eqref{eq:PZ} is given by

\begin{equation}\label{eq:denom}
\prod_{\cm \in \ZZ_2^n} \left(1- \chi_\VV(\cm)t\right)  =  \prod_{h=0}^n \big(1-(n-2h)t\big)^{\binom{n}{h}}
\end{equation}
In particular, when $n = 3$,
$$\prod_{\cm \in \ZZ_2^3} \left(1- \chi_\VV(\cm)t\right)  = (1-3t)(1-t)^3(1+t)^3(1+3t) = (1-9t^2)(1-t^2)^3.$$
Below we display the numerators for the various choices of $\arm \in \ZZ_2^3$
and the Poincar\'e series $\ms^{\arm}(t)$.   For elements of $\ZZ_2^3$ with the same
Hamming weight,  the numerators are the same,  as are the Poincar\'e series, so we  list
a single representative for each Hamming weight.  On the right we indicate the corresponding OEIS   label.

\begin{align}\label{eq:Z23}
\ms^{(000)}(t) = \displaystyle{\frac{(1-7t^2)(1-t^2)^2}{(1-9t^2)(1-t^2)^3 }} &=1+3t^2+21t^4+183t^6 + 1641 t^8 +\cdots & \hspace{-.5cm}  (\mathrm{A054879})   \nonumber   \\
\ms^{(100)}(t) = \displaystyle{\frac{t(1-3t^2)(1-t^2)^2}{(1-9t^2)(1-t^2)^3 }}  & = t+7t^3+61t^5+547t^7 + \cdots  & \hspace{-.5cm} (\mathrm{A066443})  \nonumber  \\
\ms^{(110)}(t) = \displaystyle{\frac{2t^2(1-t^2)^2}{(1-9t^2)(1-t^2)^3 }}  & = 2t^2+20t^4+182t^6 + 1640t^8 +  \cdots & \hspace{-.5cm}  (\mathrm{A125857})  \nonumber  \\
\ms^{(111)}(t) = \displaystyle{\frac{6t^3(1-t^2)^2}{(1-9t^2)(1-t^2)^3 }}  & = 6t^3+60t^5 +546 t^7 + \cdots
& \hspace{-.5cm}  (\mathrm{A054880})  \nonumber
 \end{align}
The  numbers appearing as coefficients in these series are the subscripts of the element $\arm$ in
the Bratteli diagram, and the exponent of $t$ indicates the level.

In the $n = 3$ example, the minimum polynomial of the
adjacency matrix $\Ar$ is  $p(t) = (t^2-1)(t^2-9) = t^4 - 10 t^2 + 9$, so by
\eqref{eq:recursion},   the multiplicities satisfy the recursion relation
$$\ms_{k+4}^\arm -10 \ms_{k+2}^\arm + 9\ms_k^\arm = 0$$
for all $k \geq 0$.      For example, when $\arm = (110)$,
$$\ms_8^{(110)} -10 \ms_6^{(110)} + 9\ms_4 ^{(110)} = 1640 - (10 \times 182) + (9 \times 20) = 0.$$

\medskip

Next we apply Theorem \ref{T:expgen} to determine information about  the exponential generating function
$\gr^\arm(t) = \displaystyle{\sum_{k \geq 0} \ms_{k}^{\arm}\,\frac{t^{k}}{k\,!}}$   for the multiplicities $\ms_k^\arm$
(i.e., for the number of  walks of $k$ steps from $\mathbf{0}$ to $\arm \in \ZZ_2^n$ on the $n$-cube) for
arbitrary $n$.

The eigenvalues of the adjacency matrix $\Ar$ of
the $n$-cube  are
$n-2h$, \ $h=0, 1, \ldots, n$, and  the  vectors $\cE_{\arm}$ in Corollary \ref{C:cube1}
 with $h(\arm) = h$  are the eigenvectors corresponding to the eigenvalue $n-2h$.   Since the
$\cE_{\arm}$, $\arm \in \ZZ_2^n$ form a basis for $\CC^{2^n}$,
it follows that  the minimal polynomial of $\Ar$ is
\begin{align*}p(t) &=  \prod_{h=0}^n (t-n+2h)\\
&= \begin{cases}
t(t^2-4)(t^2-16) \cdots (t^2-n^2), & \text{if \ $n$ \ is even,} \\
(t^2-1)(t^2-9) \cdots (t^2-n^2) & \text{if \ $n$ \ is odd.}
\end{cases}
\end{align*}
Suppose  $p(t)= t^{n+1} + p_{n}t^n + \cdots + p_1 t  + p_0 =  \prod_{h=0}^n (t-n+2h)$.
Then we know by Theorem \ref{T:expgen}  that $\gr^\arm(t)$ satisfies the differential equation
\begin{equation}\label{eq:deq} y^{(n+1)} + p_n y^{(n)} + \cdots + p_1 y^{(1)} + p_0 y  = 0,
\end{equation} and a  general solution of  \eqref{eq:deq} is a linear combination of the following
exponential functions
\begin{align*} 
&\er^{\pm nt}, \er^{\pm (n-2)t}, \ldots, \er^{\pm 2t}, 1& &\text{if \ $n$ \ is even} \\
&\er^{\pm nt}, \er^{\pm (n-2)t}, \ldots, \er^{\pm t}& &\text{if \ $n$ \ is odd.}
\end{align*}

\begin{thm}\label{T:expogen}
Let $\gr^{\arm}(t)$
be the exponential generating function for the number  $\ms_{k}^{\arm}$
of walks on the $n$-cube of $k$ steps starting at $\mathbf{0}$ and ending at $\arm$ (equivalently,  for the
dimension  of the irreducible modules $\Zs_k^{\arm}$ for the centralizer algebra $\Zs_k(\ZZ_2^n) = \End_{\ZZ_2^n}(\VV^{\ot k})$ for $k \geq 0$).
Then,
$$\gr^{\arm} (t) ={\sum_{k \geq 0} \ms_{k}^{\arm}\,\frac{t^{k}}{k\,!}} =\left( \cosh\,t\right)^{n-h}\left(\sinh\,t\right)^h$$
where $h = h(\arm)$, the Hamming weight of $\arm$,  and $\cosh\,t$ and $\sinh\,t$ are hyperbolic cosine and sine. 
\end{thm}
\begin{proof}
As we have noted earlier,  $\ms_k^\arm = \ms_k^\br$ for all $k\geq 0$ whenever $h(\arm) = h(\br)$, so it
suffices to assume  $\arm = (1,\dots,1,0,\dots,0) \in \ZZ_2^n$,  where there are $h = h(\arm)$ ones.
The multiplicity $\ms_{k}^{\arm}$ is the number of (ordered) $k$-tuples
$\alpha =(\alpha_1, \ldots, \alpha_k)\in [1,n]^k$ such that
$$
\ve_{\alpha_1} + \cdots + \ve_{\alpha_k}=\arm.
$$
For each $j=1, \ldots, h$,  the  number of $\alpha_i$ in $\alpha$  equal to $j$  is odd,
and for each $j=h+1, \ldots, n$, the number of $\alpha_i$ in  $\alpha$ equal to $j$ is even.
Thus, the multiplicity  is obtained by the following computation:
\begin{equation} \label{eq:mmk}
\ms_{k}^{\arm}
=  \sum_{r=1}^n\sum \, \frac{1}{\ell^\lam_1!\,\ell^\lam_2! \,\dots}\,\,{\binom{k}{\lam_1}} {\binom{k-\lam_1}{\lam_2}} \ \cdots \ {\binom{\lam_r}{\lam_r}}  \times h\,! \times \frac{(n-h)!}{(n-r)!},
 \end{equation}
where the second sum is over all partitions $\lambda=(\lambda_1, \ldots, \lambda_r)$ of $k$ with $r$ nonzero parts, such that
$$
 \lambda_1+\lambda_2+\cdots +\lambda_r=k, \quad \lambda_1 \ge \cdots \ge \lambda_h, \ \quad \  \lambda_{h+1}\ge \cdots \ge \lambda_r;
$$
$\lambda_1, \ldots, \lambda_h$ are odd numbers;   and $\lambda_{h+1}, \ldots, \lambda_r$ are even numbers.
The  factor $h!$ in \eqref{eq:mmk}  counts the number of ways to assign a number  from $\{1, \ldots, h\}$ to each odd number $\lambda_i$, $i=1, \ldots, h$.
Similarly, the factor ${(n-h)!}\,/\,{(n-r)!}=(n-h)\times \cdots \times \left(n-h-(r-h-1)\right)$ counts the ways  to assign a number from $\{h+1, \ldots, n\}$ to each even number $\lambda_i$, $i=h+1, \ldots, r$. 

Therefore,
\begin{equation} \label{eq:mmk2}
\frac{1}{k\,!}\ms_{k}^{\arm}
= \sum_{r=1}^n \sum \, \frac{1}{\ell^\lam_1!\,\ell^\lam_2! \,\dots}\,\,{\frac{1}{\lam_1!}} {\frac{1}{\lam_2!}} \ \cdots \ {\frac{1}{\lam_r!}}  \times h\,! \times \frac{(n-h)!}{(n-r)!}\,.
 \end{equation}
For a fixed value of  $r$,  the expression in the inner sum of \eqref{eq:mmk2} can be gotten by
summing the coefficients of products of $r-h$ different factors from
$\big(\cosh\,t\big)^{n-h} =\displaystyle{\left(\sum_{j=0}^\infty \frac{t^{2j}}{(2j)!}\right)^{n-h}}$
and $h$ different factors from
$\big(\sinh\,t\big)^{h} =\displaystyle{\left(\sum_{j=0}^\infty \frac{t^{2j+1}}{(2j+1)!}\right)^{h}}$
such that the total power of $t$ of those $r$ factors  is $k$.   Hence,
$\displaystyle{\frac{1}{k!} \, \ms_{k}^{\mathbf{\arm}}}$ equals the coefficient of $t^{k}$ in
$$
 \big(\cosh\,t\big)^{n-h} \big(\sinh\,t\big)^{h}=
\displaystyle{\left(\sum_{j=0}^\infty \frac{t^{2j}}{(2j)!}\right)^{n-h}}
\displaystyle{\left(\sum_{j=0}^\infty \frac{t^{2j+1}}{(2j+1)!}\right)^h}.
$$
\end{proof}

In the special case that $\arm = \mathbf 0$,  Theorem \ref{T:expogen} implies  the following

\begin{cor}\label{C:expo}
Let $\gr^{\mathbf{0}}(t)=\displaystyle{\sum_{k \geq 0} \ms_{2k}^{\mathbf{0}}\frac{t^{2k}}{(2k)!}}$
be the exponential generating function for the number  $\ms_{2k}^{\mathbf{0}}$
of walks on the $n$-cube of $2k$ steps starting and ending at $\mathbf{0}$ (equivalently,  for the
dimension $\ms_{2k}^{\mathbf{0}}$  of the space of  $\ZZ_2^n$-invariants  in $ \VV^{\ot (2k)}$;
equivalently, for the dimension of the centralizer algebra $\Zs_k(\ZZ_2^n)$ for $k \geq 0$).
Then, 
\begin{equation}\label{eq:inv} 
\gr^{\mathbf{0}}(t)=\big(\cosh\,t\big)^n = \left(\frac{\er^t + \er^{-t}}{2}\right)^n  =  \frac{1}{2^n}\sum_{i=0}^n {n \choose i} \er^{(n-2i) t}
\end{equation}
\end{cor}

\begin{remark} Corollary \ref{C:expo} implies 
\begin{align*} \gr^{\mathbf 0}(t) &= \frac{1}{2^n}\sum_{i=0}^n {n \choose i} \er^{(n-2i) t}  = \frac{1}{2^n}\sum_{i=0}^n {n \choose i}
\left(\sum_{r\geq 0} \frac{ (n-2i)^r t^r} {r\,!}\right)   \\
&= \sum_{r\geq 0} \left( \frac{1}{2^n}\sum_{i=0}^n {n \choose i} (n-2i)^{r} \right) \frac{t^r}{r\,!}, \end{align*} 
so that  the number of walks of $r$ steps from $\mathbf 0$ to $\mathbf 0$ is $\ms_{r}^{\mathbf 0}  = \displaystyle{ \frac{1}{2^n}\sum_{i=0}^n {n \choose i} (n-2i)^r}$, which is 0 unless $r = 2k$ for some $k \geq 0$.  Compare  Corollary \ref{C:cube2} with $\br = \cm = \mathbf 0$.
\end{remark} 

\begin{remark}  One may define $(\cosh\, t)^x$ for any $x \in \CC$, in particular,  for $x = -1$.   In that case 
$$(\cosh\, t)^{-1} = \sum_{j\geq 0}  \mathrm{E}_j  \frac{t^j}{j\,!},$$  where $\mathrm E_j$ is the $j$th Euler number.  
Generalized Euler numbers arising from the series expansion of  $(\cosh\, t)^{-x} = (\mathsf{sech}\,t)^x$ have been
studied and shown to have connections with Stirling numbers of the first and second kind  (see for example,  \cite{L}, \cite{KJR}).   
The  $\ms_{2k}^{\mathbf 0}$ in Corollary \ref{C:expo} are examples of such generalized Euler numbers.  
\end{remark} 

\begin{section}{Other directions} \end{section}

We conclude with a few remarks on some other directions which have been investigated that are related to walks on the $n$-cube.
\medskip

\begin{remark}  Let $\mathrm{R}, \mathrm{L}$ be the raising and lowering transformations on $\mathsf{span}_\CC\{ \underline \br \mid \br \in \ZZ_2^n\}$ defined by
$$ \mathrm{R}(\underline{\br})  = \sum_{i,  h(\br+\ve_i) > h(\br)}  \underline{\br + \ve_i} \qquad  \qquad
 \mathrm{L}(\underline \br)  = \sum_{i,  h(\br+\ve_i) < h(\br)}  \underline{\br + \ve_i}, $$  
and let $ \Ar^*(\underline \br) =  \big(n-2h(\br)\big) \underline{\br}.$
  Then $\mathrm{R} + \mathrm{L} = \Ar$ (the adjacency matrix of the $n$-cube),  and $\mathrm{R}, \mathrm{L}, \Ar^*$  determine  a canonical
  basis for a copy of $\mathfrak{sl}_2$ such that  $[\mathrm{L}, \mathrm{R}] = \Ar^*$,  $[\Ar^*, \mathrm{L}] = 2\mathrm{L}$ and
  $[\Ar^*, \mathrm{R}] = -2\mathrm{R}$.     The transformations  $\Ar, \Ar^*$ form a tridiagonal pair and generate the Terwilliger algebra
  of the $n$-cube. The details can be found in  \cite{G}.   
 \end{remark} 
 
 \begin{remark}  In  \cite{DG2}, Diaconis and Graham  considered the Markov chain  arising from  the affine 
walk on the $n$-cube given by $X_r = C X_{r-1} + \epsilon_r$, with $X_r \in \ZZ_2^n$, $C$ an 
invertible matrix with entries in  $\ZZ_2$, and  $\epsilon_r$ a random vector  in $\ZZ_2^n$ of disturbance terms.  Their analysis of such walks  relies on  codes made from binomial coefficients  $\modd 2$. 
When $C$ is the lower triangular matrix of a single Jordan block corresponding to the eigenvalue 1,  and the random vector is nonzero,  the distribution of $X_r$  tends to the uniform distribution on $\ZZ_2^n$.   Without the random vector, $X_r$ is a deterministic walk that returns to the starting point  in $n$ steps.  (See also \cite{DGM} for more on
random walks on the $n$-cube.)
\end{remark}

\medskip

 \noindent \textit{\small Department of Mathematics, University of Wisconsin-Madison, Madison, WI 53706, USA}\\
{\small benkart@math.wisc.edu}

\noindent \textit{\small Department of Mathematics, Sejong University,
Seoul, 133-747,  Korea  (ROK)}\\
{\small dhmoon@sejong.ac.kr}

 \end{document}